\documentclass[a4paper,11pt]{amsart}
\usepackage[T1]{fontenc}
\usepackage[foot]{amsaddr}
\usepackage{geometry}
\usepackage{amsfonts}
\usepackage{mathrsfs}  
\usepackage{amsthm}
\usepackage{amssymb}
\usepackage{amsmath}
\usepackage{comment}
\usepackage{theoremref}
\usepackage{enumerate}
\usepackage{bbm}
\usepackage{bm}
\usepackage{comment}
\usepackage{mathtools}
\usepackage{a4wide}
\usepackage{xcolor}
\usepackage{lipsum}

\makeatletter
\newcommand{\proofpart}[2]{%
    \par
  \addvspace{\medskipamount}%
  \noindent\emph{Step #1: #2}\par\nobreak
  \addvspace{\smallskipamount}%
  \@afterheading
}
\makeatother

\DeclarePairedDelimiter\abs{\lvert}{\rvert}%
\DeclarePairedDelimiter\norm{\lVert}{\rVert}%

\makeatletter
\let\oldabs\abs
\def\abs{\@ifstar{\oldabs}{\oldabs*}}
\let\oldnorm\norm
\def\norm{\@ifstar{\oldnorm}{\oldnorm*}}
\makeatother

\makeatletter
\g@addto@macro\bfseries{\boldmath}
\makeatother

\newcommand{\C}{\mathbb{C}}
\newcommand{\Dy}{\mathcal{D}}

\newcommand{\T}{\mathbb{T}}

\newcommand{\Ka}{\mathcal{K}}
\newcommand{\conj}[1]{\overline{#1}}
\newcommand{\D}{\mathbb{D}}

\newcommand{\Po}{\mathcal{P}}

\newcommand{\cD}{\conj{\mathbb{D}}}

\newcommand{\dist}[2]{\text{dist}( #1, #2 ) }

\newcommand{\hil}{\mathcal{H}}

\renewcommand{\Dy}{\mathcal{D}}

\renewcommand\Re{\operatorname{Re}}

\newcommand{\supp}[1]{\text{supp}({#1})}

\newtheorem{thm}{Theorem}[section]
\newtheorem{lemma}[thm]{Lemma}

\newtheorem{cor}[thm]{Corollary}
\newtheorem{prop}[thm]{Proposition}

\theoremstyle{definition}

\theoremstyle{definition}

\newtheorem{remark}[thm]{Remark}

\newcommand{\Addresses}{{
		\bigskip
		\footnotesize
         Linus Bergqvist, \\ \textsc{Stockholm, Sweden} \\
        \texttt{linus.lidman.bergqvist@gmail.com}

        \medskip
		
		Adem Limani, \\ \textsc{Centre for Mathematical Sciences, \\ Lund University, Sweden }\\
		\texttt{adem.limani@math.lu.se}
		
		\medskip
		Bartosz Malman, \\ \textsc{Institutionen för ekonomi och matematik, \\ M\"alardalen University, V\"aster\aa s, Sweden} \\
		\texttt{bartosz.malman@mdu.se}

	}}

\begin{document}
\title{\textbf{Uniqueness sets in weighted $\ell^2$-spaces via simultaneous approximation}} 

\author{Linus Bergqvist, Adem Limani, Bartosz Malman} 

\date{\today}

\begin{abstract}
    We consider a family of uniqueness problems in the setting of weighted $\ell^2$ Fourier spaces on the unit-circle, from the perspective of simultaneous approximation. We develop a general potential theoretical framework, which enables us to extend and unify several classical results by Ahlfors, Beurling, Frostman and Khrushchev, previously known in the more classical regimes. In addition, we prove a surprisingly sharp result that illustrates a deep discrepancy between unilateral and bilateral uniqueness sets.
\end{abstract}
\thanks{The second author acknowledges his financial support from the Knut \& Alice Wallenberg Foundation (grant no. 2021.0294). The third author acknowledges his financial support from from Vetenskapsrådet (VR2024-03959).}
\maketitle

\section{Background}\label{SEC:INTRO}

\subsection{Uniqueness problems on small exceptional sets}
Given a number $0\leq \alpha \leq 1$, we denote by $\hil^{-\alpha}$ the Dirichlet--Sobolev space of distributions $u$ on the unit-circle $\T \cong (-\pi,\pi]$ satisfying
\[
\sum_{n\in \mathbb{Z}} \frac{\abs{\widehat{u}(n)}^2}{(1+|n|)^{\alpha}} < \infty.
\]
Here the sequence $\{\widehat{u}(n)\}_{n\in \mathbb{Z}}$ denotes the Fourier coefficients of $u$, defined as
\[
\widehat{u}(n) := u\left(e^{it} \mapsto e^{-int} \right), \qquad n\in \mathbb{Z}.
\]
In other words, we identify $\hil^{-\alpha}$ as isometrically isomorphic with the weighted $\ell^2(w)$-space with $w(n)=(1+|n|)^{-\alpha}$ via its sequence of Fourier coefficients. These fractional Dirichlet--Sobolev spaces play a central role in function theory, potential theory, and in geometric measure theory, see \cite{MatillaFourier}. Indeed, a classical theorem of Frostman asserts that they can be utilized in computing the Hausdorff dimension of a compact set $E\subset \T$ as follows:
\[
\textbf{dim}_H(E) := \sup\{ \alpha>0: \, \exists \mu \in M_+(E) : \, \,  \mu \in \hil^{-(1-\alpha)} \},
\]
where $M_+(E)$ denotes the set of positive finite Borel measures supported on $E$. 

If $\mu \in M(\T)$ is real-valued, then a straightforward expansion involving Fourier series yields the so-called $(1-\alpha)$-Riesz energy of $\mu$:
\[
\sum_{n\in \mathbb{Z}} \frac{\abs{\widehat{\mu}(n)}^2}{(1+|n|)^{\alpha}} \asymp \int_{\T} \int_{\T} \frac{1}{|\zeta-\xi|^{1-\alpha}} d\mu(\zeta) d\mu(\xi) ,
\]
for $0<\alpha<1$. It is of interest to describe supporting sets for positive measure belonging to $\hil^{-\alpha}$. For compact sets $E\subset \T$ we define the associated $\alpha$-Riesz capacity as the quantity
\[
\operatorname{Cap}_{\alpha}(E) := \sup_{\substack{\mu \in M_+(E) \\
\mu(E)=1}} \left( \int_{\T} \int_{\T} \frac{1}{|\zeta-\xi|^{\alpha}} d\mu(\zeta) d\mu(\xi) \right)^{-1}.
\]
In the case $\alpha=1$, we have for real-valued $\mu \in M(\T)$ the comparability
\[
\sum_{n\in \mathbb{Z}} \frac{\abs{\widehat{\mu}(n)}^2}{1+|n|} \asymp \int_{\T} \int_{\T} \log \frac{2}{|\zeta-\xi|} d\mu(\zeta) d\mu(\xi),
\]
and thus retrieve the notion of the so-called \emph{logarithmic capacity}
\[
\operatorname{Cap}_0(E) := \sup_{\substack{\mu \in M_+(E) \\
\mu(E)=1}} \left( \int_{\T} \int_{\T} \log \frac{2}{|\zeta-\xi|} d\mu(\zeta) d\mu(\xi) \right)^{-1}.
\]

Now a classical result of Beurling asserts that it suffices to assume that there exists a distribution with the same property.

\begin{thm}[Beurling, \cite{beurling1949spectres}]\thlabel{THM:BEURLING} Let $0\leq \alpha <1$. If a compact set $E\subset \T$ supports a non-trivial distribution $s$ with 
\[
\sum_{n\in \mathbb{Z}} \frac{\abs{\widehat{s}(n)}^2}{(1+|n|)^{1-\alpha}} <\infty,
\]
then $\operatorname{Cap}_{\alpha}(E)>0$, and thus $E$ also supports a probability measure with the same property.
\end{thm}

A compact set $E$ which does not support a non-trivial distribution $s \in \hil^{-\alpha}$ is said to be a \emph{uniqueness set for} $\hil^{-\alpha}$. In fact, Beurling's theorem may be interpreted as a result within the category of uniqueness problems on Fourier series. It relates to the so-called Piatetski--Shapiro phenomenon, which was discovered in \cite{Piatetski-shapiro}: 

\textit{There exists a compact set $E\subset \T$ which supports a non-trivial distribution $u$ with
\[
\widehat{u}(n) \to 0, \qquad |n|\to \infty,
\]
but $E$ supports no non-trivial (Rajchman) measure on $\T$ with the same property.}

More recently, Lev and Olevskii proved that the same type of phenomenon also occurs in the setting of $\ell^p$ for $p>2$, see \cite{LevOlevskiiPS}. From this point of view, Beurling's Theorem asserts that the Piatetski--Shapiro phenomenon does not occur in a certain range of weighted $\ell^2$-spaces. 

However, if we restrict our attention to measures, then a classical result by Rajchman asserts that if $\mu\in M(\T)$ with $\widehat{\mu}(n) \to 0$ as $n\to \infty$ (or as $n\to -\infty$), then we automatically retrieve a bilateral spectral decay:
\[
\widehat{\mu}(n) \to 0, \qquad |n|\to \infty.
\]
More quantitative results of this nature were further obtained by de Leeuw and Katznelson in \cite{deleeuw1970two}. Closer to our setting of Dirichlet--Sobolev spaces, it turns out that a similar symmetric rigidity phenomenon occurs in $\hil^{-1}$:

\begin{thm}[Khrushchev--Peller] If $\mu \in M(\T)$ has positive Fourier frequencies satisfying
\[
\sum_{n>0} \frac{\abs{\widehat{\mu}(n)}^2}{1+n}< \infty,
\]
then in fact we have a bilateral summability $\mu \in \hil^{-1}$.
\end{thm}

Phrased differently, if the Cauchy integral 
\[
\Ka(\mu)(z) := \int_{\T} \frac{d\mu(\zeta)}{1-\conj{\zeta}z}= \sum_{n\geq 0} \widehat{\mu}(n)z^n, \qquad z\in \D
\]
of a measure $\mu \in M(\T)$ belongs to $\hil^{-1}$, then so does the measure $\mu$ itself. However, as observed in \cite{kozma2013singular}, in the setting of $\hil^{-\alpha}$ for $0<\alpha<1$, this result fails \footnote{Nikolaos Chalmoukis has shown the authors an independent proof of this claim, exploiting the sharpness of isoperimetrical inequality involving the embedding of the Hardy space $H^1$ into (the Bergman space) $\hil^{-1}$.}. One may still wonder if unilateral Fourier summability in $\hil^{-\alpha}$ still ensures bilateral summability, at the level of supporting sets. We will address this question later.

The general philosophy in potential theory is that the aforementioned notions of capacity typically characterize exceptional sets for their associated dual spaces, for example negligible zero sets and sets of singularities. To clarify this point, we recall that the dual space of $\hil^{-\alpha}$ may be identified with the Dirichlet-type spaces $\hil^{\alpha}$ consisting of $f\in L^2(\T,dm)$ satisfying
\[
\sum_{n\in \mathbb{Z}} \abs{\widehat{f}(n)}^2 (1+|n|)^{\alpha} < \infty.
\]
Considering Poisson extensions of elements $f\in \hil^{\alpha}$ to the unit-disc $\D=\{|z|<1\}$:
\[
P(f)(z) = \int_{\T} \frac{1-|z|^2}{|\zeta-z|^2} f(\zeta) dm(\zeta), \qquad z\in \D,
\]
we may isomorphically identify $\hil^{\alpha}$ with the space of harmonic functions in $\D$ satisfying
\[
\int_{\D} \abs{\nabla P(f)(z)}^2 (1-|z|)^{1-\alpha} dA(z) < \infty,
\]
where $dA$ denotes the unit-normalized area measure on $\D$. A classical result of Beurling asserts that for any $0\leq \alpha <1$, there exists a constant $C(\alpha)>0$, such that the following weak-type estimate holds for capacities:
\[
\operatorname{Cap}_{\alpha} \left( \{\zeta \in \T: \abs{f(\zeta)}> \lambda \} \right) \leq C(\alpha) \frac{\norm{f}^2_{\hil^{1-\alpha}}}{\lambda^2} \qquad \lambda>0, \qquad f\in C(\T) \cap \hil^\alpha.
\]
The above weak-type estimate lies at the heart of showing that elements in $\hil^{\alpha}$ are well-defined modulo a set of zero $\alpha$-capacity. Furthermore, Carleson proved in \cite{carlesonuniqueness} that compact sets $E\subset \T$ of zero $\alpha$-capacity support functions in $\hil^\alpha$ with blow-up precisely there, and which vanish precisely there, clarifying their role as exceptional sets.

\subsection{Uniqueness problems on large exceptional sets}
The problem of uniqueness sets also makes sense in the dual framework. More specifically, for $\alpha \in (0,1]$ one may ask to describe which compact sets $E \subset \T$ support a non-trivial function $f\in \hil^{\alpha}$. For $\alpha=1$, the result traces back to the early works of Ahlfors and Beurling in the 1950s, where they obtained the following characterization. 

\begin{thm}[Ahlfors--Beurling, \cite{ahlfors1950conformal}] \thlabel{THM:A-BTHM} Let $E \subset \T$ be a compact set of positive Lebesgue measure. Then the following statements are equivalent:
\begin{enumerate}
    \item[(i)] $E$ supports no non-trivial element in $\hil^1$.
    \item[(ii)] Every function $f$ analytic in $\C \setminus E$ with 
    \[
    \int_{\C \setminus E} \abs{f'(z)}^2 dA(z) < \infty,
    \]
    is constant.
    \item[(iii)] For any arc $I\subseteq \T$, we have
    \begin{equation*}\label{EQ:ABCOND}
    \operatorname{Cap}_0(I\setminus E) = \operatorname{Cap}_0(I).
    \end{equation*}

\end{enumerate}
\end{thm}

The property in $(ii)$ is related to the concept of removable singularities for analytic functions with finite Dirichlet integral in the complex plane, which may be interpreted as Riemann Removability Theorem specific to the class of analytic functions with finite planar Dirichlet integral. Admittedly, Ahlfors and Beurling studied the question of removability of general sets in $\C$, not necessarily confined inside $\T$, and their result contained a characterization involving the notion of extremal lengths, which could be recaptured as the capacity in the setting of $\T$. Extensions of these ideas and generalizations to non-linear potential theoretical settings were later carried out by L. Hedberg in \cite{hedberg1974removable}. 

Notably, it is intriguing that the notion of log-capacity, somewhat of an intrinsic device that captures the exceptional sets in the $\hil^{-1}$, again turns out to be the right apparatus for describing the exceptional sets in the dual space $\hil^1$. 

As seen from the Khrushchev--Peller Theorem, the problem of the Cauchy integral of a measure belonging to $\hil^{-1}$ is equivalent to the measure itself belonging there. Surprisingly, it turns out that in the dual setting of $\hil^{-1}$, this symmetry rigidity fails even at the level of support sets, further highlighting a stark discrepancy between the unilateral and the bilateral uniqueness problem in $\hil^1$.

\begin{thm}[Khrushchev, \cite{khrushchev1978problem}] \thlabel{THM:KHRUSHUNIBI} There exists a compact set $E\subset \T$ of positive Lebesgue measure, which supports a complex finite Borel measure $\mu$ satisfying the unilateral estimate
\[
\sum_{n>0} \abs{\widehat{\mu}(n)}^2 (1+n) < \infty,
\]
but $E$ supports no non-trivial function $f \in \hil^1$. 
\end{thm}
In other words, unlike in the case of $\hil^{-1}$ illustrated by the Khrushchev-Peller Theorem stated above, there is an asymmetry phenomenon on support sets for unilateral and bilateral average Fourier decay in $\hil^1$. Khrushchev's proof is primarily based on his diligent work on uniqueness sets for Cauchy transforms, which he characterized in terms of entropy. He then proceeds to construct a compact set $E$ of positive Lebesgue measure, which violates the capacity condition $(iii)$ in the Ahlfors--Beurling Theorem, yet has finite Beurling--Carleson entropy:
\[
\sum_I \abs{I} \log \frac{1}{|I|} < \infty,
\]
where $\{I\}$ are the connected components of $\T \setminus E$. As such, it relies on deep potential theory, and  N. G. Makarov in \cite{makarov1991class} provided some slightly simplifying perspectives to the matter. On a related note, a strong discrepancy between the unilateral and the bilateral Fourier uniqueness problem in the setting for $\ell^p$, was recently shown in \cite{limani2026asymmetric}.

\subsection{Uniqueness problems in weighted $\ell^2$-spaces}

Our purpose and aims in this paper are two-fold. First, we aim to develop a general potential theoretical framework in order to extend several of the aforementioned classical results to general weighted $\ell^2(\omega)$-spaces, for a wide range of sequences $\omega= \{\omega(n)\}_n$. Secondly, we show that all these problems can, in fact, be unified under the common umbrella of \emph{simultaneous approximation}, which captures phenomena of the following kind:
\\ \\
\emph{There exists a family of functions $\{\phi_j\}_j$ (analytic/trigonometric polynomials, smooth functions, etc) which converge to $1$ in $\ell^2(\omega)$, but $\phi_j$ (and perhaps also its derivatives) tend to $0$ uniformly on an exceptional set $E$.
}
\\ \\
This work is in part inspired from S. Khrushchev in \cite{khrushchev1978problem}, and also from L. Hedberg in \cite{hedberg1974removable}.
\\
We now disclose the details of the setting we consider. For us, a weight $W$ will be a non-negative continuously differentiable function on $(0,1]$, which is integrable on $[0,1]$ with respect to $dx$. Given a weight $W$, we shall throughout these notes denote its associated moment sequence by
\[
\omega(n) := \int_0^1 r^n W(1-r) dr, \qquad n=0,1,2,\dots
\]
Note that the integrability of $W$ and the dominated convergence theorem ensure that 
\[
\omega(n) \downarrow 0, \qquad n\to \infty.
\]
Consider the Hilbert space $h^2(W)$ of harmonic functions $u$ in the unit-disc $\D$ normed by
\[
\|u\|^2_{h^2(W)} := \int_{\D} \abs{u(z)}^2 W(1-|z|^2) dA(z)< \infty.
\]
It is easy to see that the dilation $u_r(z):=u(rz)$ converge to $u$ in $h^2(W)$ as$r \to 1-$. We may therefore define the Fourier coefficients of $u$ via 
\[
\widehat{u}(n) := \lim_{r\to 1-} \widehat{u_r}(n) = \lim_{r\to 1-} \int_{\T} u(r\zeta) \zeta^{-n} dm(\zeta), \qquad n\in \mathbb{Z}, 
\]
where each $u_r$ is a harmonic function in a neighborhood of the closed disc $\cD$. With this at hand, we can now re-identify $h^2(W)$ as isometrically isomorphic with the Hilbert space $\ell^2(\omega)$ of two-sided infinite sequences of Fourier coefficients satisfying
\[
\norm{u}_{\ell^2(\omega)}^2 = \sum_{n\in \mathbb{Z}} \abs{\widehat{u}(n)}^2 \omega(|n|) = \int_{\D} \abs{u(z)}^2 W(1-|z|^2) dA(z).
\]
This connection has the advantage that we may freely identify elements in $\ell^2(\omega)$ with those in $h^2(W)$, and vice versa, which captures an environment where potential theory meets harmonic analysis. Now in order to give rise to fruitful theory, we shall need to impose some natural conditions on our weights $W$, which translate to conditions on the associated weight sequences $\omega$. Our conditions are:

\begin{enumerate}
    \item[]{$(C_1)$} $W$ is either monotone non-decreasing or monotone non-increasing,
    \item[]{$(C_2)$} the Dini-condition for $W$ fails, i.e., we have
    \[
    \int_0^1 \frac{W(t)}{t}dt =+ \infty.
    \]
    \item[]{$(C_3)$} $W$ satisfies the doubling condition:
    \begin{equation*}
    W(t)\leq c\, W(t/2), \qquad 0<t<1,
\end{equation*}
for some $0<c<2$ depending only on $W$.
\end{enumerate}

The condition $(C_1)$ is a relatively mild condition on $W$, and is satisfied by essentially all concrete weights of considerable interest. Technically, the monotonicity assumption is only relevant in a neighborhood of the origin, but for lack of significance, and simplicity, we shall assume monotonicity on all of $(0,1]$.

The condition $(C_2)$ is imposed in order to avoid redundancies in building the potential theoretical framework for exceptional sets in $\ell^2(\omega)$. Indeed, it is easily seen to be equivalent to the $\ell^1$-divergence of the moments:
\[
\sum_{n\geq 0} \omega(n) =+ \infty,
\]
which simply excludes that $\ell^2(\omega)$ contains atomic measures, making finite sets of points have zero capacity.

Let $\Omega$ denote the reciprocal weight: 
\[
\Omega(n):= \frac{1}{\omega(n)} \qquad n=0,1,2,\dots
\] Then the dual space of $\ell^2(\omega)$ can be identified with $\ell^2(\Omega)$ (see Section~\ref{SEC:CauchyDualitySec}). If the integral in $(C_2)$ converges, then Cauchy-Schwarz inequality implies that $\ell^2(\Omega)$ consists of continuous functions on $\T$. This also makes the notion of supporting sets in the dual space $\ell^2(\Omega)$ rather redundant. \\

The condition $(C_3)$ is a mild regularity condition. It enables some flexibility when moving between scales at a controlled fashion, but importantly prohibits $\omega$ from decaying too fast. Indeed, together with $(C_1)$ we get that $W(1/n)\gtrsim \frac{1}{n^{\alpha}}$ for some $0<\alpha<1$, which at its turn leads to the following moment decay of the associated moment sequence 
\[
\omega(n) \gtrsim \frac{1}{n^{1+\alpha}}, \qquad n=1,2,\dots
\]
However, such rapid decay is already excluded by the condition $(C_2)$, hence it is intrinsically the synergy between $(C_1)-(C_3)$ that carries the theory around. In fact, the principal role of the conditions $(C_1)-(C_3)$ is that we are able to obtain a Douglas representation formula for the dual space $\ell^2(\Omega)$, see \thref{THM:DOUGLASFORMULA} below. 

Below, we have summarized a range of weights covered by our framework. It can be seen that the associated moment sequences $\omega$ may essentially decay as slowly as possible, and on the other end be as close as possible to $\ell^1$-summability. The precise relationship between $W$ and $\omega$ will be clarified in Section~$3$. 

\[
\renewcommand{\arraystretch}{2.5}
\begin{array}{c|c|c|c}
\boldsymbol{W(t)}
&
\boldsymbol{\omega(n)}
\\
\hline
\displaystyle
W(t)\asymp
\frac{1}{
t\log(e^{e}/t)
\bigl(\log\log(e^{e}/t)\bigr)^2
}
&
\displaystyle
\omega(n)\asymp
\frac{1}{\log\log n}
\\
\hline
\displaystyle
W(t)\asymp
\frac{1}{t\bigl(\log(e/t)\bigr)^2}
&
\displaystyle
\omega(n)\asymp
\frac{1}{\log n}
\\
\hline
\displaystyle
W(t)\asymp t^{\alpha-1},
\quad 0 < \alpha <1
&
\displaystyle
\omega(n)\asymp n^{-\alpha}
\\
\hline
\displaystyle
W(t)\asymp
\frac{1}{\log(e/t)}
&
\displaystyle
\omega(n)\asymp
\frac{1}{n\log n}

\end{array}
\]

Our main results are gathered in the next section, and are divided into subsections featuring their content.

\section{Main results}

\subsection{Small exceptional sets}
We start out with gathering our results related to the problem of uniqueness sets in $\ell^2(\omega)$. First out is the following vast generalization of Beurling's \thref{THM:BEURLING} to the setting of $\ell^2(\omega)$.

\begin{thm}\thlabel{THM:Distrib} Let $W$ be a weight satisfying the hypothesis $(C_1)-(C_3)$, and $\{\omega(n)\}_{n\geq 0}$ be its associated moment sequence. If a compact set $E\subset \T$ supports a non-trivial distribution $S$ with
\[
\sum_{n\in \mathbb{Z}} \abs{\widehat{S}(n)}^2 \omega(|n|) < \infty,
\]
then $E$ also supports a probability measure $\mu$ with $\{\widehat{\mu}(n)\}_{n\in \mathbb{Z}} \in \ell^2(\omega)$.
\end{thm}

Notably, our result confirms the idea that the Piatetski-Shapiro phenomenon never occurs in weighted $\ell^2$-spaces. This is a generalization of results in \cite{kozma2013singular} and \cite{lev2011wiener} for standard weights $\omega(n) = n^{-\alpha}$, $\alpha \in (0,1]$. As seen from the diagram of concrete weights satisfying $(C_1)-(C_3)$ the statement holds for a wide range of moment sequences $\{\omega(n)\}_n \in c_0 \setminus \ell^1$.

Moving forward, we next illustrate that the uniqueness problem in $\ell^2(\omega)$ exhibits a rather neat symmetry. In particular, it shows that the Khrushchev--Peller Theorem at least remains true at the level of supporting sets.

\begin{thm}\thlabel{THM:UNILAT-BILAT} Let $W$ be a weight satisfying the hypothesis $(C_1)-(C_3)$, and $\{\omega(n)\}_{n\geq 0}$ be its associated moment sequence. If a compact set $E\subset \T$ supports a non-trivial measure $\mu$ whose positive frequencies satisfy
\begin{equation} \label{EQ:OneSided}
\sum_{n > 0} \, \abs{\widehat{\mu}(n)}^2 \omega(n) < \infty,
\end{equation}
then $E$ also supports a probability measure $\mu$ with $\{\widehat{\mu}(n)\}_{n\in \mathbb{Z}} \in \ell^2(\omega)$.
\end{thm}

Surprisingly, this result also appears to be new even for standard weights $\omega(n)= n^{-\alpha}$ for $0<\alpha<1$. Meanwhile, as we shall see, the case $\alpha=1$ can somewhat be attributed to the recent work in \cite{beneteau2020simultaneous}.

The proof is based on developing an appropriate potential theoretical framework in the setting of $\ell^2(\omega)$ where we introduce a notion of capacity which classifies the different notions of exceptional sets. The so-called $\omega$-capacity of a compact subset $E\subset \T$ is defined as a quantity comparable to
\begin{equation}
    \label{DEF:CAPW}
    \operatorname{Cap}_{\omega}(E) \asymp \sup_{\substack{\mu \in M_+(E) \\ \mu(E)=1 }} \frac{1}{\norm{\mu}^2_{\ell^2(\omega)}}.
\end{equation}

We refer the reader to \eqref{EQ:omegaCAP} in Section~\ref{SEC:PotTheorFramework} for a precise definition. In particular, a compact set $E\subset \T$ has zero $\omega$-capacity if every probability measure $\mu$ supported in $E$ satisfies
\[
\sum_{n\in \mathbb{Z}} \abs{\widehat{\mu}(n)}^2 \omega(|n|)=+\infty.
\]

\subsection{Problems of analytic approximation and translates}

Here, we present an aspect of our result, viewed from a dual perspective related to the phenomenon of simultaneous approximation and the problem of translations.

From now and onward, we shall let $\Omega:= \{\Omega(n)\}_{n=0}^\infty$ denote the reciprocal weight of $\omega$, defined by
\[
\Omega(n) = \frac{1}{\omega(n)} \qquad n=0,1,2,\dots
\]
Note that we have $\Omega(n) \uparrow +\infty$, since $\omega(n) \downarrow 0$.  By $\ell^2(\Omega)$, we shall designate the Hilbert space of elements $f\in L^2(\T,dm)$ with
\[
\sum_{n\in \mathbb{Z}} \, \abs{\widehat{f}(n)}^2 \Omega(|n|)< \infty.
\]

We denote by $\ell_A^2(\Omega)$ the closed subspace of $\ell^2(\Omega)$ of functions whose Poisson extension to $\D$ are analytic. Phrased differently, we have
\[
\ell^2_A(\Omega):= \{f\in \ell^2(\Omega): \widehat{f}(n)=0, \quad n<0 \}.
\]
Let also $C_A(\T)$ denote the space of continuous functions $f$ in $\T$ whose Poisson extension to $\D$ is analytic.

Our next result shows that the exceptional sets in $\ell^2(\omega)$ also characterize various other phenomena related to approximation with analytic functions.

\begin{thm} \thlabel{Thm: four equivalent properties} Let $W$ be a weight satisfying the hypothesis $(C_1)-(C_3)$, and $\{\omega(n)\}_{n\geq 0}$ be its associated moment sequence. For a compact set $E\subset \T$, the following statements are all equivalent:
\begin{enumerate}
    \item[(i)] $E$ has zero $\omega$-capacity.

    \item[(ii)] There exists an outer $f\in C_A(\T) \cap \ell_A^2(\Omega)$ with $f=0$ on $E$, such that
    \[
    \{f(z)z^n: n=0,1,2,\dots \}
    \]
    has a dense linear span in $\ell^2_A(\Omega)$. Moreover, $f$ can be chosen to have an analytic continuation to a neighbourhood of every point of $\T \setminus E$.

    \item[(iii)] For any $h \in \ell_A^2(\Omega)$ and $g\in C(E)$, there exists analytic polynomials $(Q_n)_n$ such that 
    \[
    \norm{Q_n-h}_{\ell^2(\Omega)} + \sup_{\zeta \in E} \abs{Q_n(\zeta)-g(\zeta)} \to 0.
    \]

\end{enumerate}
\end{thm}
Note that $(iii)$ is a simultaneous approximation phenomenon of analytic polynomials in $\ell^2_A(\Omega)$. The condition in $(ii)$ is a statement about cyclicity of the unilateral shift on $\ell^2_A(\Omega)$, asserting that $E$ is an admissible zero set for a cyclic vector. It is tempting to conjecture that the statement in $(ii)$ can be upgraded to:

\emph{Every non-trivial outer function $f \in C_A(\T) \cap \ell^2_A(\Omega)$ with $f=0$ only in $E$ is a cyclic function for the unilateral shift in $\ell^2_A(\Omega)$. }

For $\Omega(n)=n$, this is essentially the content of the Brown--Shields conjecture, see \cite{brown1984cyclic}.

Another classical problem inspired by Wiener's problem on translates, that has previously been viewed in the classical settings of Dirichlet spaces, is to consider the action of translations 
\[
\{\widehat{f}(m)\}_{m\in \mathbb{Z}} \mapsto \{\widehat{f}(m+n)\}_{m\in \mathbb{Z}}, \qquad n\in \mathbb{Z}
\]
on elements in $\ell^2(\Omega)$. This corresponds to the action of multiplication $f \mapsto \conj{\zeta}^n f$ on the spatial side, and a central question is to describe for which functions $f \in \ell^2(\Omega)$ the linear span of the set
\[
\{f(\zeta)\zeta^n: n\in \mathbb{Z} \}
\]
form a dense subset in $\ell^2(\Omega)$. Such elements are said to be $\emph{bicyclic}$, stemming from the fact that one considers bilateral translations of $f$. 

Our next result extends a classical theorem of Wiener to our setting.

\begin{thm}\thlabel{THM:BICYCLIC}
Let $W$ be a weight satisfying the hypothesis $(C_1)-(C_3)$, and $\{\omega(n)\}_{n\geq 0}$ be the associated moment sequence, and $\Omega(n) = \frac{1}{\omega(n)}$. For compact sets $E\subseteq \T$, the following statements are all equivalent:
\begin{enumerate}
    \item[(i)] $E$ has zero $\omega$-capacity.
    \item[(ii)] Every function $f\in C(\T)\cap \ell^2(\Omega)$ with zero set
    \[
    \{\zeta \in \T: f(\zeta)=0\} \subseteq E
    \]
    is bicyclic in $\ell^2(\Omega)$.
    \item[(iii)] The space $C^\infty_0(\T \setminus E)$ of smooth functions compactly supported in $\T \setminus E$ is dense in $\ell^2(\Omega)$.
\end{enumerate}

\end{thm}

Note that the condition $(iii)$ may again be interpreted as a simultaneous approximation phenomenon with smooth functions. From a slightly different point of view, since $C_0^\infty(\T \setminus E)$ is translation invariant, the condition $(iii)$ may also be interpreted as $E$ being a set which allows an approximate reconstruction of the space $\ell^2(\Omega)$, from (smooth) functions that vanish in a neighborhood of $E$.

\subsection{A duality principle on exceptional sets}

In a similar vein, we now ask which compact sets $E\subset \T$ of positive Lebesgue measure support a non-trivial function $f$ with $f\in \ell^2(\Omega)$. Our main result shows that the notion of $\omega$-capacity introduced in \eqref{DEF:CAPW} can still be used to yield a complete characterization.


\begin{thm}\thlabel{THM:SUPPDIRICHLET}
Let $\{\omega(n)\}_{n\geq 0}$ be the moment sequence of a weight $W$ satisfying the hypotheses $(C_1)-(C_3)$.
Let $\Omega(n) = 1/\omega(n)$, $n \geq 0$, be the associated dual weight. Then the following statements on a compact subset $E\subseteq \T$ of positive Lebesgue measure are all equivalent: 
\begin{enumerate}
    \item[(i)] $E$ supports no non-trivial $f \in \ell^2(\Omega)$,
    \item[(ii)] For any arc $I\subset \T$:
    \begin{equation*}
\operatorname{Cap}_\omega(I \setminus E) = \operatorname{Cap}_\omega(I).
\end{equation*}
\item[(iii)] \[\operatorname{Cap}_\omega(\T\setminus E)= \operatorname{Cap}_\omega(\T).\]
\end{enumerate}
\end{thm}

We note the interesting global-local phenomenon that is demonstrated by the equivalence of $(ii)$ and $(iii)$. The proof is again based on a simultaneous approximation argument. In fact, we shall in Section \ref{SSEC:BATHM} show that \thref{THM:SUPPDIRICHLET} leads to a new and simpler proof of the Ahlfors--Beurling Theorem. \\

In this regime of supporting sets of functions in $\ell^2(\Omega)$, we saw in \thref{THM:KHRUSHUNIBI} that for $\Omega(n) = n$, there exist sets which support a measure whose positive Fourier frequencies are $\ell^2(\Omega)$-summable, but which support no non-trivial element in $\ell^2(\Omega)$. In our next result, we are able to maximally extend Khrushchev's \thref{THM:KHRUSHUNIBI} on the discrepancy between unilateral and bilateral Fourier decay for supporting sets in $\ell^2(\Omega)$.

\begin{thm}\thlabel{THM:IMPKHRU} For any non-decreasing sequence of positive real numbers $\{\Omega(n)\}_{n\geq 0}$ tending to $+\infty$ and satisfying 
\begin{equation}\label{EQ:OMEGACOND}
\sup_{n> 1} \frac{\Omega(n)}{\log n} = +\infty,
\end{equation}
there exists a compact set $E\subset\T$ of positive Lebesgue measure such that:
\begin{enumerate}
    \item[(i)] For any non-trivial measure $\mu \in M(E)$ we have
    \[
    \sum_{n\in \mathbb{Z}} \abs{\widehat{\mu}(n)}^2 \Omega(|n|) =+\infty.
    \]
    \item[(ii)] There exists a non-trivial $\mu \in M(E)$ with
    \[
    \widehat{\mu}(n) = \mathcal{O}\left(n^{-M}\right), \qquad n\to +\infty,
    \]
    for any $M>0$.
\end{enumerate}
\end{thm}

Surprisingly, it turns out that the condition in \eqref{EQ:OMEGACOND} is very sharp in the following precise sense. Whenever a compact set $E\subset \T$ satisfies the property $(ii)$ in the statement of \thref{THM:IMPKHRU}, there exists a compact subset $K\subseteq E$ of positive Lebesgue measure for which its indicator function satisfies:
\[
\sum_n \log (1+|n|) \abs{\widehat{1_K}(n)}^2 < \infty.
\]
This will be proved in \thref{PROP:ENTROPYDECAY}.



\subsection{Notation and organization}
The manuscript is organized as follows: 

Section~\ref{SEC:PotTheorFramework} will be devoted to developing a potential theoretical framework in the context of $\ell^2(\omega)$. We introduce a notion of energy, potentials, and $\omega$-capacity, which turns out to be the appropriate measure for our exceptional sets. At the end of the section, we shall deduce a Douglas-type representation formula for the Cauchy dual of $\ell^2(\omega)$, which will play a crucial role for our developments. 

Section~\ref{SEC:MAINTECHRES} is devoted to the proofs of our results announced in Sections 2.1--2.2, namely Theorems 2.1--2.4. There we will crucially make use of our potential theoretical framework and develop the common simultaneous approximation schemes, allowing us to give a unified approach to the aforementioned results.

Lastly, Section~\ref{SEC:UNIQDIRICHLET} is devoted to proving our results on uniqueness sets in $\ell^2(\Omega)$ in \thref{THM:SUPPDIRICHLET}, and to proving a sharp discrepancy between supporting sets for unilateral versus bilateral Fourier decay in \thref{THM:IMPKHRU}. The section concludes with an application of our developments in order to give a short and simple proof of the Ahlfors--Beurling \thref{THM:A-BTHM}.

If two quantities $A,B>0$ are related by $A\leq cB$ for some constant $c>0$, we simply write $A\lesssim B$. In case both $A\lesssim B$ and $B\lesssim A$ hold, we simply write $A\asymp B$. Occasionally, it will be convenient to allow constants to change from line to line, despite maintaining the same designation.

\section{A potential theoretical framework}\label{SEC:PotTheorFramework}

\subsection{Potentials and $\omega$-capacity}

Our goal in this subsection is to develop a potential-theoretic framework in the setting of $\ell^2(\omega)$, where $\omega=\{\omega(n)\}_{n\geq 0}$ is a moment sequence associated with a radial weight $W$ satisfying the conditions $(C_1)$--$(C_3)$. The development of potential-theoretical frameworks has also previously been carried out in other similar frameworks — for instance, by Chalmoukis and Hartz in the setting of Drury-Arveson spaces \cite{chalmoukis2024potential}. See also \cite{chalmoukis2022totally} for work in similar directions.

Given a positive finite Borel measure $\mu$ on $\T$, we let
\[
P(\mu)(z) := \int_{\T} \frac{1-|z|^2}{\abs{\zeta-z}^2} \, d\mu(\zeta), \qquad z \in \D,
\]
denote its harmonic Poisson extension to the unit disc $\D$. To compute the $\ell^2(\omega)$-norm of $\mu$, or equivalently the $h^2(W)$-norm of $P(\mu)$, we use Fubini--Tonelli to obtain
\begin{align*}
\sum_{n\in\mathbb{Z}} \abs{\widehat{\mu}(n)}^2 \omega(|n|)
&= \int_{\D} P(\mu)(z)^2 W(1-|z|^2)\,dA(z) \\
&= \int_{\T}\int_{\T}\int_{\D}
\frac{(1-|z|^2)^2W(1-|z|^2)}
{\abs{z-e^{it}}^2\abs{z-e^{is}}^2}
\,dA(z)\,d\mu(e^{it})\,d\mu(e^{is}) \\
&= \int_{\T}\int_{\T} K_W(e^{it},e^{is})
\,d\mu(e^{it})\,d\mu(e^{is}),
ademlimz
\end{align*}
where the positive kernel $K_W:\T\times\T\to(0,\infty]$ is defined by the area integral
\[
K_W(e^{it},e^{is})
:= \int_{\D}
\frac{(1-|z|^2)^2W(1-|z|^2)}
{\abs{z-e^{it}}^2\abs{z-e^{is}}^2}
\,dA(z).
\]
The expression for $K_W$ can be simplified substantially. Indeed, using polar coordinates, we may write
\[
K_W(e^{it},e^{is})
= \int_0^1\int_0^{2\pi}
\frac{(1-r^2)^2W(1-r^2)}
{|re^{i\theta}-e^{it}|^2|re^{i\theta}-e^{is}|^2}
\frac{d\theta\,rdr}{\pi}.
\]
The inner integral can be computed using the reproducing property of the Poisson kernel as follows
\begin{align*}
\int_0^{2\pi}
\frac{(1-r^2)^2}
{|re^{i\theta}-e^{it}|^2|re^{i\theta}-e^{is}|^2}
\frac{d\theta}{2\pi}
&= \int_0^{2\pi}
P_{re^{is}}(e^{i\theta})P_{re^{it}}(e^{i\theta})
\frac{d\theta}{2\pi} \\
&= P_{r^2e^{is}}(e^{it}).
\end{align*}
With this at hand, we may now apply the change of variables $r\mapsto r^2$, which yields the simplified expression
\begin{align*}
K_W(e^{it},e^{is})
&= \int_0^1 P_{re^{is}}(e^{it})W(1-r)\,dr \\
&= \int_0^1
\frac{1-r^2}{1-2r\cos(t-s)+r^2}W(1-r)\,dr.
\end{align*}
This formula shows that $K_W$ depends only on $t-s$ and satisfies the symmetry properties
\[
K_W(e^{it},e^{is})
=K_W(e^{i(t-s)},1)
=K_W(e^{i(s-t)},1),
\]
Moreover, the function $t \mapsto K_W(e^{it},1)$ is positive and even on $[-\pi,\pi]$, and non-increasing on $[0,\pi]$. This implies that $K_W(e^{it},e^{is})$ actually only depends on the distance $|t-s|$, hence we shall whenever convenient also use the designation
\[
K_W(e^{is},e^{it}) = \kappa_W(|s-t|) \qquad e^{it}, e^{is} \in \T. 
\]

Let us record the Fourier representation of the kernel. For fixed $\zeta\in\T$, the function $e^{it}\mapsto K_W(e^{it},\zeta)$ belongs to $L^1(\T)$, since
\[
\int_{\T}K_W(e^{it},\zeta)\,dt
= \int_0^1\left(\int_{\T}P_{r\zeta}(e^{it})\,dt\right)W(1-r)\,dr
= \int_0^1W(r)\,dr<\infty.
\]
Hence we can compute its Fourier coefficients in the usual way, which are given by
\[
\widehat{K_W(\cdot,\zeta)}(n)
= \omega(|n|)\conj{\zeta}^{n}, \qquad n\in\mathbb{Z}.
\]
We define the $\omega$-energy of a positive finite Borel measure $\mu$ on $\T$ by
\begin{equation}
    \label{E:omegaEnergyDef}
    \mathcal{E}_\omega(\mu)
:= \int_{\T}\int_{\T}\kappa_W(|s-t|)
\,d\mu(e^{it})\,d\mu(e^{is}).
\end{equation}
The preceding computation shows that
\[
\mathcal{E}_\omega(\mu)
= \sum_{n\in\mathbb{Z}}\omega(|n|)\abs{\widehat{\mu}(n)}^2 = \omega(0)\mu(\T)^2
+2\sum_{n\geq 1}\omega(n)\abs{\widehat{\mu}(n)}^2.
\]
Hence $\mu$ has finite $\omega$-energy if and only if its sequence of Fourier coefficients belongs to $\ell^2(\omega)$. Furthermore, we also note that the condition $(C_2)$, the divergence of the Dini-integral of $W$, can be rephrased as the condition $K_W(0+)=+\infty$. This further clarifies that Dirac measures have infinite $\omega$-energy.

Given a compact set $E\subset\T$, let $P(E)$ denote the set of probability measures supported on $E$. We define the $\omega$-capacity of $E$ by
\begin{equation}\label{EQ:omegaCAP}
\operatorname{Cap}_\omega(E)
:= \sup_{\mu\in P(E)}\frac{1}{\mathcal{E}_\omega(\mu)}.
\end{equation}

We may also extend the notion of $\omega$-capacity to open sets $U\subset \T$ by 
\[
\operatorname{Cap}_\omega (U) = \sup \left\{ \operatorname{Cap}_\omega(E): E \subseteq U, \, E \text{ compact} \right\}.
\]

It is also useful to introduce the potentials associated with $K_W$. Given a positive finite Borel measure $\mu$ on $\T$, we define its $W$-potential by
\[
U_\mu(e^{it})
:= \int_{\T}\kappa_W(|s-t|)\,d\mu(e^{is}).
\]
If $\mu$ has finite $\omega$-energy, then $U_\mu$ is finite $\mu$-almost everywhere. Indeed, this follows from the identity
\[
\int_{\T}U_\mu\,d\mu=\mathcal{E}_\omega(\mu).
\]
Formally, $U_\mu=K_W*\mu$, and the convolution formula gives
\[
\widehat{U_\mu}(n)
= \widehat{K_W}(n)\widehat{\mu}(n)
= \omega(|n|)\widehat{\mu}(n).
\]

Since the energy $\mathcal{E}_\omega$ is defined in terms of a continuous even kernel and monotone kernel $K_W$, our $\omega$-capacity enjoys many of the properties of the classical logarithmic capacity. In particular, we shall use the following facts proved in \cite[Chapter 2.1]{dirichletspaceprimer}.

\begin{lemma} \thlabel{LEM:CapacityProperties} \,
\begin{enumerate}[(i)]
    \item (Outer regularity) If $\{E_n\}_{n \geq 1}$ is a sequence of decreasing compact sets satisfying $\cap_{n} E_n = E$, then 
    \[
    \lim_{n \to \infty} \operatorname{Cap}_\omega(E_n) = \operatorname{Cap}_\omega(E).
    \]
    \item (Existence of equilibrium measures) If a compact set $E\subset \T$ satisfies $\operatorname{Cap}_\omega(E) > 0$, then there exists $\nu\in P(E)$ such that 
    \[
    \operatorname{Cap}_\omega(E) = \frac{1}{\mathcal{E}_\omega(\nu)}.
    \]
    \item (Frostman's Theorem for equilibrium potentials) Let $\nu$ be an equilibrium measure for a compact set $E$ satisfying $\operatorname{Cap}_\omega(E) > 0$. Then the equilibrium potential $U_\nu$ satisfies
    \begin{itemize}
    \item $U_\nu(x) \leq 1/\operatorname{Cap}_\omega(E)$ for all $x \in \text{supp}(\nu)$,
    \item $U_\nu(x) \geq 1/\operatorname{Cap}_\omega(E)$ for all $x \in E \setminus E_0$, where $E_0 \subset E$ is some set of zero (outer\footnote{In the sense that $\inf_{U} \operatorname{Cap}_\omega(U) = 0$, where the infimum is taken over open sets $U$ containing $E_0$.}) $\omega$-capacity.
    \end{itemize}
\end{enumerate}
    
\end{lemma}

For later use, it will also be convenient to consider the analytic Herglotz average of $K_W$, defined as
\begin{equation}
\label{E:KWADEF}
K_W^A(z,\zeta)
:= \int_0^1\frac{\zeta+rz}{\zeta-rz}W(1-r)\,dr,
\quad
z\in\mathbb{C}\setminus\{\zeta/r:r\in(0,1]\},
\quad \zeta\in\T.
\end{equation}
It is easy to see that $z\mapsto K_W^A(z,\zeta)$ is analytic in the slit plane $\mathbb{C}\setminus\{\zeta/r:r\in(0,1]\}$. Moreover, for $z\in\D$,
\[
\Re K_W^A(z,\zeta)
= \int_0^1P_{rz}(\zeta)W(1-r)\,dr.
\]
At boundary points $e^{it} \in\T\setminus\{\zeta\}$, its real part agrees with $K_W(e^{it},\zeta)$. Later, we will need the following simple lemma quantifying the magnitude of $K_W^A$ away from the slits $\{\zeta/r : r \in (0, 1]\}$.

\begin{lemma} \thlabel{Lemma: K_W^A bound}
    For every number $\delta>0$, there is a constant $C_\delta>0$ such that
    \[
    |K_W^A(z, \zeta)| \leq C_\delta,
    \]
    whenever $\inf_{r \in (0,1]} |\zeta-rz| \geq \delta$, $\zeta \in \T$.
\end{lemma}
\begin{proof}
    We need to examine the kernel $(1+r z \overline{\zeta} )/(1 - r z \overline{\zeta})$. Clearly $|1+r z \overline{\zeta}| \leq 2$. Next, $|1 - r z \overline{\zeta}| = |\zeta - r z| \geq \delta$. 
    Putting these estimates together we see that
    \[
    \left| \frac{1+r z \overline{\zeta} }{1 - r z \overline{\zeta}} \right| \leq \frac{2}{\delta}.
    \]
    Thus using the definition of $K_W^A$ shows that
    \[
    |K_W^A(z, \zeta)| \leq \frac{2}{\delta} \int_0^1 W(r) dr. 
    \]
\end{proof}

\subsection{Cauchy duality}
\label{SEC:CauchyDualitySec}

Here, we clarify the concept of Cauchy duality between $\ell^2(\omega)$ and $\ell^2(\Omega)$. Here $\Omega$ is the reciprocal weight sequence of $\omega$, that is
\[
\Omega(n) := \frac{1}{\omega(n)}, \qquad n=0,1,2,\dots
\]
This duality relation plays a central role and will be extensively utilized in our developments, and only the requirement that the weight sequence satisfies $\omega(n) \downarrow 0$ as $n\to \infty$ is needed. Note by Cauchy-Schwarz inequality, we have
\[
\abs{\sum_{n\in \mathbb{Z}} \widehat{f}(n) \conj{\widehat{g}(n)} }\leq \norm{f}_{\ell^2(\omega)} \norm{g}_{\ell^2(\Omega)},
\]
hence every element in $g\in \ell^2(\Omega)$ defines a dual element of $\ell^2(\omega)$ in the classical $\ell^2$ Cauchy-pairing
\[
\ell_g(f):= \sum_{n} \widehat{f}(n) \conj{\widehat{g}(n)}, \qquad f\in \ell^2(\omega).
\]
Conversely, if $\ell$ is a bounded linear functional on the Hilbert space $\ell^2(\omega)$, then it follows from Riesz representation theorem that we can identify $\ell$ as a unique element $g\in \ell^2(\omega)$ considered in canonical Hilbert space pairing of $\ell^2(\omega)$:
\[
\ell(f) = \sum_{n} \widehat{f}(n) \conj{\widehat{g}(n)} \omega(n), \qquad f\in \ell^2(\omega).
\]
Now setting 
\[
G_w(\zeta) := \sum_n \widehat{g}(n) \omega(n) \zeta^n \qquad \zeta \in \T,
\]
we immediately see that $G_w \in \ell^2(\Omega)$ with 
\[
\norm{G_w}_{\ell^2(\Omega)} = \norm{g}_{\ell^2(\omega)},
\]
and
\[
\abs{\sum_n \widehat{f}(n) \conj{\widehat{G_w}(n)}} = \abs{\ell(f)} \leq C \norm{f}_{\ell^2(\omega)}, \qquad f\in \ell^2(\omega).
\]
This establishes the duality relation $\ell^2(\omega) \cong \ell^2(\Omega)'$ viewed in the $\ell^2$ Cauchy-pairing. Repeating the same argument for $\ell^2(\Omega)$ we actually see that $\ell^2(\omega)$ is reflexive. Furthermore, since $\ell^2(\Omega) \hookrightarrow \ell^2$, we may also express the Cauchy pairing as
\[
\sum_n \widehat{f}(n) \conj{\widehat{g}(n)} = \lim_{r\to 1-} \int_{\T} f(r\zeta) \conj{g(\zeta)} dm(\zeta), \qquad f \in \ell^2(\omega), \quad g\in \ell^2(\Omega).
\]
In the intermediate step, we used that dilatations with $0<r<1$
\[
f_r(\zeta) := \sum_n \widehat{f}(n) r^{|n|} \zeta^n, \qquad \zeta \in \T,
\]
of elements $f\in \ell^2(\omega)$ are continuous in $\ell^2(\omega)$.

\subsection{Moment sequences and elements in $\ell^2(\omega)$}

Here, we derive some asymptotic behaviors of the associated moment sequences of weights $W$. This will ensure that the elements in our range of spaces $\ell^2(\omega)$ may be identified as genuine distributions. 

As briefly mentioned in the introduction, only the behavior of $W$ close to the origin matters, as for any fixed $0<\delta<1$:
\[
\int_{0}^{1-\delta} r^n W(1-r) dr \leq (1-\delta)^n \int_0^1 W(r)dr \to 0, \qquad \text{exponentially}.
\]

The following lemma establishes a rather precise asymptotic behavior of the moment sequence, more explicitly expressed in terms of $W$.
\begin{lemma}\thlabel{LEM:MOMW} Let $W$ be a weight satisfying the doubling estimate
\begin{equation}\label{EQ:W<Double}
    W(t)\leq cW(t/2), \qquad 0<t<1,
\end{equation}
for some $c=c(W)>0$. Then the associated moment sequence satisfies 
\[
\omega(n) \asymp \int_0^{1/n} W(t) dt.
\]
\end{lemma}
Note that \eqref{EQ:W<Double} is substantially weaker than $(C_3)$, as no quantitative size on $c(W)>0$ is imposed.
\begin{proof}
Set 
\[
A(t) := \int_{0}^t W(s) ds, \qquad 0\leq t\leq 1,
\]
which is increasing. The lower bound simply follows from the estimates
\[
\omega(n) \geq \int_0^{1/n} (1-t)^n W(t) dt \geq (1-1/n)^n A(1/n) \geq \frac{A(1/n)}{4}, \quad n \geq 2.
\]
It thus only remains to establish the reverse estimate. To this end, we primarily note that the doubling estimate in \eqref{EQ:W<Double} on $W$ implies that
\[
A(t) = 2\int_0^{t/2} W(2s) ds \leq 2c A(t/2), \qquad 0<t<1.
\]
With this at hand, we may bound the integral by splitting it up into smaller intervals. Namely, set $J = \lfloor\log_2 n \rfloor$ and note that
\[
\omega(n) = \int_0^{1/n}(1-t)^{n} W(t)dt + \sum_{1\leq j\leq J} \int_{2^{j-1} /n}^{2^{j}/n} (1-t)^n W(t)dt + \int_{2^J/n}^{1} (1-t)^n W(t)dt.
\]
Now the first term is trivially 
\[
\int_0^{1/n}(1-t)^{n} W(t)dt \leq A(1/n).
\]

For the terms appearing in the sum, we use the estimate $(1-t)^n \leq e^{-nt}$, in conjunction with the doubling estimate on $A$. This yields
\[
\int_{2^{j-1} /n}^{2^{j}/n} (1-t)^n W(t)dt \leq e^{-2^{j-1}} A(2^j/n) \leq (2c)^j e^{-2^{j-1}} A(1/n), \qquad 1\leq j \leq \log_2 n.
\]
Since the sum in $j$ is convergent, we conclude that
\[
\sum_{1\leq j\leq J} \int_{2^{j-1} /n}^{2^{j}/n} (1-t)^n W(t)dt \lesssim A(1/n).
\]

Finally, for the last term we may again use the estimate $(1-t)^n \leq e^{-nt}$ together with the fact that $nt \geq 2^J$ throughout this interval to obtain
\[
\int_{2^J/n}^{1} (1-t)^n W(t)dt \leq e^{-2^J} A(1).
\]
Applying the doubling estimate $J+1$ times together with the fact that $2^{-(J+1)} \leq 1/n$ shows that $A(1) \leq (2c)^{J+1}A(1/n)$, which together with the other estimates finishes the proof.
\end{proof}

We remark that the above lemma applies to non-increasing weights, as well as non-decreasing weights which do not decay too rapidly. Furthermore, it is easy to see that for monotone $W$ satisfying \eqref{EQ:W<Double}, the associated moment sequence satisfies the decay restriction 
\[
\omega(n) \gtrsim \frac{1}{n^\alpha},\qquad n\geq 1,
\]
for some $\alpha=\alpha(W)>0$. This implies that elements $u \in \ell^2(\omega) \cong h^2(W)$ may be regarded as distributions on $\T$, that is
\[
\abs{\widehat{u}(n)} \lesssim (1+|n|)^{\alpha} , \qquad n=0,\pm1,\pm 2, \dots
\]
This has the advantage that the $\ell^2$ Cauchy-pairing of $u$ with a smooth $\phi \in C^\infty(\T)$ can also be expressed as the usual distributional pairing:
\[
\conj{u}(\phi) = \sum_{n} \conj{\widehat{u}(n)} \widehat{\phi}(n), \qquad \phi \in C^\infty(\T).
\]
By duality, it then obviously follows that $C^\infty(\T) \subset \ell^2(\Omega)$.

\subsection{A Douglas formula for $\ell^2(\Omega)$} 
In this subsection, we shall establish a Douglas-type  representation formula for $\ell^2(\Omega)$, where $\Omega$ denotes the reciprocal weight of $\omega$. Our main result can be summarized as follows.

\begin{thm}\thlabel{THM:DOUGLASFORMULA} Let $W$ be a weight satisfying the conditions $(C_1)-(C_3)$ and let $\omega$ be its associated moment sequence, and let $\Omega$ be the reciprocal of $\omega$. Setting 
\[
h(t) := \frac{W(t)}{\left( \int_0^t W(s) \, ds \right)^2} \qquad 0<t<1,
\]
we have that the following norm of equivalence holds:
\[
\sum_{n\in \mathbb{Z}} \abs{\widehat{f}(n)}^2 \, \Omega(|n|) \asymp \int_{-\pi}^{\pi}\int_{-\pi}^{\pi} \abs{f(e^{it})-f(e^{is})}^2 h\left(\frac{|s-t|}{\pi}\right) ds dt, 
\]
for all $f\in \ell^2(\Omega)$ with zero average $\widehat{f}(0)=0$.
\end{thm}

Our Douglas-type representation formula holds for a wide range of moment sequences $\{\omega(n)\}_n$ corresponding to radial weights $W$. Below, we have gathered a list of standard weights for which one can compare and easily verify the relationships of the various quantities. Notably, our theory essentially covers the regime of all moment sequences $\omega(n)$ that range from decaying arbitrarily slowly, to being close to the threshold of being summable.

\[
\renewcommand{\arraystretch}{2.5}
\begin{array}{c|c|c|c}
\boldsymbol{W(t)}
&
\boldsymbol{\omega(n)}
&
\boldsymbol{h(t)}
\\
\hline
\displaystyle
W(t)\asymp
\frac{1}{
t\log(e^{e}/t)
\bigl(\log\log(e^{e}/t)\bigr)^2
}
&
\displaystyle
\omega(n)\asymp
\frac{1}{\log\log n}
&
\displaystyle
h(t)\asymp
\frac{1}{t\log(e^{e}/t)}
\\
\hline
\displaystyle
W(t)\asymp
\frac{1}{t\bigl(\log(e/t)\bigr)^2}
&
\displaystyle
\omega(n)\asymp
\frac{1}{\log n}
&
\displaystyle
h(t)\asymp
\frac{1}{t}
\\
\hline
\displaystyle
W(t)\asymp t^{\alpha-1},
\quad 0\leq \alpha <1
&
\displaystyle
\omega(n)\asymp n^{-\alpha}
&
\displaystyle
h(t)\asymp t^{-1-\alpha}
\\
\hline
\displaystyle
W(t)\asymp
\frac{1}{\log(e/t)}
&
\displaystyle
\omega(n)\asymp
\frac{1}{n\log n}
&
\displaystyle
h(t)\asymp
\frac{\log(e/t)}{t^2}
\end{array}
\]

The remaining part of this subsection will thus be devoted to the proof of \thref{THM:DOUGLASFORMULA}, and some relevant properties of spaces $\ell^2(\Omega)$ that exhibit a Douglas-type representation formula as above.

Now let $h$ be a non-negative function, and note that by making the change of variables $x=s-t$ and $y=t$, we have by Parseval's Theorem that 

\begin{multline*}
\int_{-\pi}^{\pi} \int_{-\pi}^{\pi} \abs{f(e^{it})-f(e^{is})}^2 h\left(\frac{|s-t|}{\pi}\right) ds dt = \sum_{n\neq 0 } \abs{\widehat{f}(n)}^2 \int_{-\pi}^{\pi} \abs{e^{inx}-1}^2 h\left(\frac{|x|}{\pi}\right) dx \\
= 8\sum_{n\neq 0}\abs{\widehat{f}(n)}^2 \int_0^{\pi} \sin^2(nx/2) h(x/\pi) dx.
\end{multline*}
In view of this expansion, the proof of \thref{THM:DOUGLASFORMULA} readily reduces to establishing the following asymptotics.

\begin{prop}\thlabel{PROP:hasymp}  Let $W$ be a weight satisfying the conditions $(C_1)-(C_3)$ and let $\omega$ be its associated moment sequence, and let $\Omega$ be the reciprocal of $\omega$. Setting 
\[
h(t) := \frac{W(t)}{\left( \int_0^t W(s) \, ds \right)^2} \qquad 0<t<1,
\]
the following asymptotic holds 
\[
\Omega(n) \asymp \int_0^{\pi} \sin^2(nt/2) h(t/\pi) dt, \qquad n\neq 0.
\]
\end{prop}

The proof will be divided into the distinct cases that can occur under the assumption $(C_1)$, namely when $W$ is assumed to be non-increasing, and when $W$ is assumed to be non-decreasing. As it is a bit long, we have decided to divide it into lemmas, treating the two different cases of $W$, separately. A common observation in both cases is that the assumptions $(C_1)-(C_3)$ imply that 
\[
\lim_{t\to 0+} \frac{W(t)}{t} =+ \infty.
\]
Indeed, the claim is trivial if $W$ is non-increasing. If $W$ is non-decreasing, then \thref{LEM:MOMW} implies that
\[
\frac{1}{n^{1+\alpha}} \lesssim \omega(n) \asymp \int_0^{1/n} W(t) dt \leq W(1/n) 1/n,
\]
for some $0<\alpha<1$, from which the claim easily follows.

\begin{lemma}\thlabel{LEM:Wdec} Let $W$ be a weight which is non-increasing and satisfies
\[
\lim_{t\to 0+} \frac{W(t)}{t}=+ \infty.
\]
Then the following comparability holds:
\[
\int_0^{\pi} \sin^2(nt/2) h(t/\pi) dt \asymp \frac{1}{\int_0^{1/n}W(t)dt}, \qquad n=1,2,\dots
\]
with
\[
h(t) = \frac{W(t)}{\left(\int_0^t W(s)ds\right)^2}
\]
\end{lemma}
 
\begin{proof}
For simplicity, set
\[
A(t) = \int_0^t W(s)ds,
\]
and note that the assumption on $W$ implies that $A$ is non-decreasing and concave, and thus $-1/A$ is also non-decreasing and concave. This implies that
\[
h(t) = \frac{d}{dt}\left( - \frac{1}{A(t)}\right)
\]
is positive and non-increasing. 
\proofpart{1}{Lower bound:}
Partition the interval $[\frac{\pi}{2n},\pi]$ into the arcs
\[
J_k=\left(\frac{k\pi}{2n},\frac{(k+1)\pi}{2n}\right],
\qquad k=1,\ldots,2n-1.
\]
Observe that on every fourth arc, we have
\[
\sin^2(nt/2)\geq\frac{1}{2}, \qquad t\in J_{1+4k}.
\]
Since $h(t/\pi)$ is non-increasing, the first arc $J_{1+4k}$ contains at least $\frac{1}{4}$ of the mass of $4$ consecutive arcs to the right of $J_{1+4k}$, we get by summing up that 
\[
\begin{aligned}
\int_0^\pi\sin^2(nt/2)h(t/\pi)\,dt
&\geq\frac{1}{8}\int_{\frac{\pi}{2n}}^\pi h(t/\pi)\,dt\\
&=\frac{\pi}{8}
\left(
\frac1{A(1/2n)}-\frac1{A(1)}
\right)\\
&\gtrsim\frac1{A(1/n)}.
\end{aligned}
\]

\proofpart{2}{Upper bound:}
For the upper bound, note that
\[
\int_{0}^{\pi} \sin^2(nt/2) h(t/\pi) dt \leq n^2\int_{0}^{\pi/n}  t^2 h(t/\pi) dt + \int_{\pi/n}^{\pi} h(t/\pi) dt.
\]
The second integral trivially satisfies the required upper bound, as seen from
\[
\int_{\pi/n}^{\pi} h(t/\pi) dt = \pi \int_{1/n}^1 h(t) dt = \pi \left( \frac{1}{A(1/n)}- \frac{1}{A(1)} \right) \leq \frac{\pi}{A(1/n)}.
\]
Upon changing variables, it remains only to show that
\[
n^2 \int_0^{1/n} t^2 h(t)dt \lesssim \frac{1}{A(1/n)}.
\]
To this end, note that an integration by parts gives
\[
\int_0^{1/n}t^2 h(t) dt = - \frac{1}{n^2 A(1/n)} +  \lim_{\varepsilon\to 0+}\frac{\varepsilon^2}{A(\varepsilon)} + 2 \int_0^{1/n}\frac{t}{A(t)} dt.
\]
Since $W$ is non-increasing, the assumption $W(t)/t \to +\infty$ as $t\to 0+$ implies that the limiting boundary term vanishes
\[
0\leq \frac{\varepsilon^2}{A(\varepsilon)}\leq \frac{\varepsilon}{W(\varepsilon)} \to 0, \qquad \varepsilon\to 0+.
\]
Furthermore, the concavity of $A$ implies that 
\[
A(1/n) \leq \frac{A(t)}{tn}, \qquad 0<t< \frac{1}{n},
\]
and hence 
\[
\int_0^{1/n}\frac{t}{A(t)} dt \leq \frac{1}{A(1/n)} \frac{1}{n} \int_0^{1/n}dt \leq \frac{1}{A(1/n)} \frac{1}{n^2}.
\]
Combining these estimates, we arrive at
\[
\int_0^{\pi} \sin^2(nt/2) h(t/\pi)dt \lesssim \frac{1}{A(1/n)}.
\]
This concludes the proof of the lemma.
\end{proof}

We now proceed with proving the same asymptotics, but for $W$ which are non-decreasing but do not grow too rapidly.

\begin{lemma}\thlabel{LEM:Winc} Let $W$ be a non-decreasing weight, which satisfies $(C_3)$ and $W(t)/t \to +\infty$ as $t\to 0+$.
Then we have
\[
\int_0^\pi \sin^2(nt/2) h(t/\pi) dt \asymp \frac{1}{A(1/n)},
\]
with $h(t)=\frac{d}{dt}\left(-\frac{1}{A(t)}\right)$.
\end{lemma}

\begin{proof}
This time, since $W$ is non-decreasing, the function $A$ is convex. As $A(0)=0$, it follows that
\[
A(t/2)\leq \frac{1}{2}A(t).
\]
\proofpart{1}{Lower bound:}
For the lower bound, note that we have
\[
\sin^2(nt/2)\geq \frac{1}{2},
\qquad
\frac{\pi}{2n}\leq t\leq \frac{\pi}{n}.
\]
Using this, we get
\[
\begin{aligned}
\int_0^\pi \sin^2(nt/2)h(t/\pi)\,dt
&\geq
\frac{1}{2}
\int_{\pi/(2n)}^{\pi/n}h(t/\pi)\,dt\\
&=
\frac{\pi}{2}
\int_{1/(2n)}^{1/n}h(t)\,dt\\
&=
\frac{\pi}{2}
\left(
\frac{1}{A(1/(2n))}
-
\frac{1}{A(1/n)}
\right)
&\geq
\frac{\pi}{2} \frac{1}{A(1/n)},
\end{aligned}
\]
where we in the last step utilized the convexity of $A$ as:
\[
A(1/(2n))\leq \frac{1}{2}A(1/n).
\]
This shows that
\[
\int_0^\pi \sin^2(nt/2)h(t/\pi)\,dt
\gtrsim
\frac{1}{A(1/n)}.
\]
\proofpart{2}{Upper bound:}
For the upper bound, we again use the pointwise estimate
\[
\sin^2(nt/2)\leq \min(n^2t^2,1)
\]
and a simple change of variable to obtain
\[
\int_0^\pi \sin^2(nt/2)h(t/\pi)dt
\lesssim
n^2\int_0^{1/n}t^2h(t)\,dt
+
\int_{1/n}^1h(t)\,dt.
\]
The second term is easily controlled
\[
\int_{1/n}^1h(t)\,dt
=
\frac{1}{A(1/n)}
-
\frac{1}{A(1)}
\leq
\frac{1}{A(1/n)}.
\]
For the first term, we again carry out an integration by parts 
\[
\int_0^{1/n}t^2h(t)\,dt
=
-\frac{1}{n^2A(1/n)}
+
\lim_{\varepsilon\to0+}
\frac{\varepsilon^2}{A(\varepsilon)}
+
2\int_0^{1/n}\frac{t}{A(t)}\,dt.
\]
Now the assumption $W(t)/t\to +\infty$ as $t\to 0+$ again ensures that the boundary term vanishes. Indeed, since $W$ is non-decreasing,
\[
A(\varepsilon)
\geq
\int_{\varepsilon/2}^{\varepsilon}W(s)\,ds
\geq
\frac{\varepsilon}{2}W(\varepsilon/2),
\]
and hence
\[
0\leq
\frac{\varepsilon^2}{A(\varepsilon)}
\leq
\frac{2\varepsilon}{W(\varepsilon/2)}
=
4\frac{\varepsilon/2}{W(\varepsilon/2)}
\longrightarrow 0.
\]
We therefore have the estimate
\[
n^2\int_0^{1/n}t^2h(t)\,dt
\leq
2n^2\int_0^{1/n}\frac{t}{A(t)}\,dt.
\]

We split the last integral into dyadic intervals:
\[
\begin{aligned}
n^2\int_0^{1/n}\frac{t}{A(t)}\,dt
&=
n^2\sum_{k\geq0}
\int_{2^{-(k+1)}/n}^{2^{-k}/n}
\frac{t}{A(t)}\,dt\\
&\lesssim
\sum_{k\geq0}
\frac{2^{-2k}}{A(2^{-(k+1)}/n)}.
\end{aligned}
\]
Observe that the doubling assumption $(C_3)$ on $W$  implies
\[
A(t) =
2\int_0^{t/2}W(2s)\,ds \leq
2C\int_0^{t/2}W(s)\,ds
=
2C A(t/2).
\]
Iterating this inequality gives
\[
A(1/n)
\leq
(2C)^{k+1}
A(2^{-(k+1)}/n),
\]
and as a consequence, we obtain
\[
\begin{aligned}
n^2\int_0^{1/n}\frac{t}{A(t)}\,dt
&\lesssim
\frac{1}{A(1/n)}
\sum_{k\geq0}
2^{-2k}(2C)^{k+1}\\
&\lesssim
\frac{1}{A(1/n)}
\sum_{k\geq0}
\left(\frac{C}{2}\right)^k\\
&\lesssim
\frac{1}{A(1/n)},
\end{aligned}
\]
where the last series converges because $C<2$. This proves the upper bound and completes the proof.
\end{proof}

We may now summarize our discussions and establish our main intention in this subsection.

\begin{proof}[Proof of \thref{PROP:hasymp}] 
Note that both \thref{LEM:Wdec} and \thref{LEM:Winc} apply under the assumptions $(C_1)-(C_3)$, which yields the asymptotics:
\[
\int_0^{\pi} \sin^2(nt/2) h(t/\pi) dt \asymp \frac{1}{\int_0^{1/n}W(t) \, dt}, \qquad n=1,2,3,\dots
\]
with 
\[
h(t) = \frac{W(t)}{\left( \int_0^t W(s) \, ds \right)^2}, \qquad 0<t<1.
\]
However, appealing to \thref{LEM:MOMW} we also get that the above quantities are comparable to $\Omega(n)$, and thus the proof is complete.
\end{proof}

We have thus established the Douglas-type representation formula for our considered family of spaces $\ell^2(\Omega)$. As a consequence of this formula, we prove useful boundedness statements for two nonlinear operations on $\ell^2(\Omega)$.

\begin{lemma}\thlabel{Lemma: max_Dirichlet}
    Suppose $\ell^2(\Omega)$ admits a Douglas-type norm representation formula as in \thref{THM:DOUGLASFORMULA}. 
    Let $f, g$ be two real-valued functions in $\ell^2(\Omega)$. Then for almost all $(\zeta, \xi) \in \T^2$, we have 
    \[
    |M_{f,g}(\zeta) - M_{f,g}(\xi)|^2 \leq \max(|f(\zeta)-f(\xi)|^2, |g(\zeta)-g(\xi)|^2),
    \]
    where $M_{f,g}= \max(f,g)$ is defined almost everywhere on $\T$. Thus
    \[
    \|M_{f,g}\|^2_{\ell^2(\Omega)}  \lesssim \|f\|^2_{\ell^2(\Omega)} + \|g\|^2_{\ell^2(\Omega)}.
    \]
    In particular, for such functions, $f, g \in \ell^2(\Omega) \implies \max(f,g) \in \ell^2(\Omega)$.
    
\end{lemma}
Note that since $\min(f,g) = - \max(-f,-g)$, the corresponding statements with $\min$ instead of $\max$ also hold. 

\begin{proof}
    
All statements follow immediately from the first inequality and the Douglas-type representation formula. Thus it suffices to show that

\[
|\max(f(\zeta),g(\zeta)) - \max\left(f(\xi),g(\xi))|^2 \leq \max(|f(\zeta)-f(\xi)|^2, |g(\zeta)-g(\xi)|^2\right),
\]

for almost all $(\zeta, \xi) \in \T^2$. But this follows immediately from the elementary inequality 
\[
|\max(a,b) - \max(c,d)|^2 \leq \max(|a-c|^2, |b-d|^2).
\qedhere
\]

\end{proof}

\begin{lemma}\thlabel{LEM:ExpLemmaDouglas}
    Suppose $\ell^2(\Omega)$ admits a Douglas-type norm representation formula as in \thref{THM:DOUGLASFORMULA}. If $g \in \ell^2(\Omega)$ satisfies $\Re g \geq 0$ on $\T$, then $G := e^{-g} \in \ell^2(\Omega)$, and there exists a numerical constant $C >0$ such that
    \[ \|G - G(0)\|_{\ell^2(\Omega)} \leq C \|g\|_{\ell^2(\Omega)}. \]
\end{lemma}

\begin{proof}
    By basic calculus we have that
    \[ |G(e^{it}) - G(e^{is})| \leq |g(e^{it}) - g(e^{is})| \cdot \sup_{z \in L} e^{\Re z},\] where $L$ is the straight line segment between $-g(e^{it})$ and $-g(e^{is})$. Since this line segment is contained in the left half-plane, we obtain $e^{\Re z} \leq 1$ for $z \in L$. Plugging this estimate into the Douglas formula in \thref{THM:DOUGLASFORMULA} gives us the desired bound.
\end{proof}

More generally, if $\Phi:\C \to \C$ is a globally Lipschitz continuous function, then using the Douglas representation formula of $\ell^2(\Omega)$, one easily shows that 
\begin{equation}\label{EQ:Lipl2Omega}
\norm{\Phi \circ f}_{\ell^2(\Omega)}\leq C(\Phi) \norm{f}_{\ell^2(\Omega)}, \qquad f\in \ell^2(\Omega).
\end{equation}

\section{Supporting sets for distributions and complex measures}\label{SEC:MAINTECHRES}

In this section we prove two theorems that show that existence of a more general object -- such as a distribution or complex measure that satisfies a one-sided estimate -- with support contained in a compact set $E$, is enough to conclude that $E$ has positive $\omega$-capacity, i.e. that $E$ supports a probability measure that satisfies a two-sided bound. Both theorems are, in slightly different ways, derived from simultaneous approximation results for sets of $\omega$-capacity $0$. 

We begin by proving \thref{THM:Distrib}, which is a result about support sets of certain distributions.

The proof will be based on the following simultaneous approximation scheme, which is largely constructive.

\begin{prop}\thlabel{PROP:SAdistGENERAL} Let $E\subset \T$ be a compact set of zero $\omega$-capacity. Then for any $\varepsilon>0$ there exists a positive harmonic function $u_\varepsilon$ on $\D$ which extends to a $C^\infty$-function on $\T$ and satisfies the following properties: 
\begin{enumerate}
    \item[(i)] $u_\varepsilon$ is compactly supported in $\T \setminus E$,
    \item[(ii)] $\|u_\varepsilon -1\|_{\ell^2(\Omega)}\leq \varepsilon$.
\end{enumerate}

\end{prop}

Before proving this Proposition, we shall derive \thref{THM:Distrib} as a consequence of \thref{PROP:SAdistGENERAL}.

\begin{proof}[Proof of \thref{THM:Distrib}] We argue by contradiction. Let $E$ be a compact set of zero $\omega$-capacity, and assume that $S$ is a non-trivial distribution supported in $E$ with 
\[
\{\widehat{S}(n)\}_{n\in \mathbb{Z}} \in \ell^2(\omega).
\]
By means of substituting $S$ with $\zeta^k S$, we may without loss of generality assume that $S(1) = \widehat{S}(0)\neq 0$. But by applying \thref{PROP:SAdistGENERAL}, we see that
\[
0=S(u_\varepsilon)= \sum_{n\in \mathbb{Z}}\widehat{S}(n) \widehat{u}_{\varepsilon}(-n) \to \widehat{S}(0), \qquad \varepsilon\to 0+,
\]
which gives the desired contradiction. 
\end{proof}

\begin{proof}[Proof of \thref{PROP:SAdistGENERAL}]
The main idea of the proof is to consider suitable neighborhoods of $E$, from which we take the corresponding equilibrium potentials, in order to construct the functions $f_\varepsilon$ with the desired simultaneous approximation property.

To this end, recall that
\[
\sum_n |\widehat{\mu}(n)|^2 \omega(|n|) \asymp \int_{\T} \int_{\T} \kappa_W(|s-t|) d\mu(e^{it}) d\mu(e^{is}),
\]
and that the $W$-potential of a positive Borel measure $\mu$ is defined as
\[
U_\mu (e^{it}) := \int_{\T}  \kappa_W(|s-t|)\, d\mu(e^{is}).
\]

According to \thref{LEM:CapacityProperties}, we can for every $\varepsilon>0$, find some closed neighborhood $E_\varepsilon \supset E$ such that $\operatorname{Cap}_\omega(E_\varepsilon) \leq \varepsilon^2$. Let $\nu$ be an equilibrium measure for $E_\varepsilon$. Then $g_\varepsilon := \operatorname{Cap}_\omega(E_\varepsilon) U_\nu$ satisfies
$$
\| g_\varepsilon\|_{\ell^2(\Omega)} =  \operatorname{Cap}_\omega(E_\varepsilon) \| U_\nu(x)\|_{\ell^2(\Omega)} = \operatorname{Cap}_\omega(E_\varepsilon) \| \nu \|_{\ell^2(\omega)}.
$$
Thus
$$
\| g_\varepsilon\|_{\ell^2(\Omega)} = \sqrt{\operatorname{Cap}_\omega(E_\varepsilon) } \leq \varepsilon.
$$

We can now define the function $h_\varepsilon := \min(1,g_\varepsilon)$. By property $(iii)$ of \thref{LEM:CapacityProperties}, it follows that $h_\varepsilon = 1$ everywhere on $E_\varepsilon \setminus F_0$, where $F_0$ is a set of $\omega$-capacity $0$. Furthermore, since $K_W$ is positive, we also know that $h_\varepsilon \geq 0$ everywhere on $\T$.

Finally, since $h_\varepsilon \leq g_\varepsilon$ everywhere on $\T$, we have that $|h_\varepsilon(0)| \leq |g_\varepsilon(0)|$, and by combining this with \thref{Lemma: max_Dirichlet} for the remaining Fourier coefficients, we get that
$$
\| h_\varepsilon\|_{\ell^2(\Omega)} \lesssim \| g_\varepsilon\|_{\ell^2(\Omega)}.
$$

Thus, the function $1-h_\varepsilon$ vanishes on $E_\varepsilon$ (except possibly on a set of $\omega$-capacity 0), and furthermore
$$
\|1-(1-h_\varepsilon) \|_{\ell^2(\Omega)} \lesssim \|g_\varepsilon \|_{\ell^2(\Omega)} = \sqrt{\operatorname{Cap}_\omega(E_\varepsilon) } \leq \varepsilon. 
$$

We now proceed to define the function $f_\varepsilon$ as a convolution of $1-h_\varepsilon$ with a smooth approximate of the identity. 

To this end, fix a number $0 < \delta < \text{dist}(E, \T \setminus E_\varepsilon)/2$, and let $\varphi_\delta$ be a smooth approximate of the identity, supported on arc of length $\delta$. We define
\[
f_\varepsilon := (1- h_\varepsilon) * \varphi_\delta.
\]
Clearly $f_\varepsilon \in C^\infty(\T)$. And since $\delta < \text{dist}(E, \T \setminus E_\varepsilon)/2$, we have that $f_\varepsilon = 0$ at every point on the set $\{\zeta \in \T: \text{dist}(\zeta, E) < \delta/2 \}$. 
Note that since sets of zero $\omega$-capacity must also have zero Lebesgue measure, the points in $F_0$ do not affect $f_\varepsilon$. It follows that $f_\varepsilon$ is compactly supported in $\T \setminus E$. Finally, we also have
\[
\| 1- f_\varepsilon \|_{\ell^2(\Omega)} = \| (1 - (1- h_\varepsilon)) * \varphi_\delta \|_{\ell^2(\Omega)} \leq \| 1 - (1- h_\varepsilon) \|_{\ell^2(\Omega)} \lesssim \varepsilon.
\]
Here we used that $\varphi_\delta dm$ is a probability measure, hence $1= \widehat{\varphi_\delta}(0)= 1 \ast \varphi_{\delta}$, and that convolving with $\varphi_\delta$ does not increase the norm, in view of the convolution formula
\[
\abs{\widehat{f \ast \varphi_\delta}(n)} = \abs{\widehat{f}(n)}\abs{\widehat{\varphi_\delta}(n)} \leq \abs{\widehat{f}(n)}, \qquad n\in \mathbb{Z}.
\]
Taking the Poisson extensions of $f_\varepsilon$ yields a harmonic function $u_\varepsilon$ on $\D$ whose restriction to the boundary $\T$ has compact support in $\T \setminus E$ and such that $\|u_\varepsilon - 1 \|_{\ell^2(\Omega)} \leq C \varepsilon$. Replacing $\varepsilon$ with $\varepsilon / C$ gives the desired formula.
\end{proof}

We shall now prove \thref{Thm: four equivalent properties} which shows the equivalence of three properties of a compact subset $E$ of $\T$. We start with two lemmas, the first of which is a generalization to the context of $\ell^2(\Omega)$-spaces of \cite[Theorem 3.4.1]{dirichletspaceprimer}, which details a construction of an analytic function in the Dirichlet space blowing up at $E$. 

\begin{lemma} \thlabel{L:BLOWUPLEMMA}
    Let $E \subset \T$ be a compact set of zero $\omega$-capacity. There exists a function $\Phi \in \ell^2_A(\Omega)$ that has an analytic continuation to a neighbourhood of every point of $\T \setminus E$ and that satisfies \[ \lim_{z \to \zeta} \Re \Phi(z) = +\infty\] for all $\zeta \in E$, and $\Re \Phi \geq 0$ on $\D$. The function $\Phi$ can be chosen to have the form \[\Phi(z) = \sum_{n \geq 0} \phi_n(z)\] with sum convergent in $\ell^2(\Omega)$, $\sum_{n \geq 0} \|\phi_n\|_{\ell^2(\Omega)} < \infty$, and where each $\phi_n$ extends analytically to a disk\footnote{Of radius depending on $n$, of course.} containing $\cD$ and having positive real part in $\D$.
\end{lemma}

\begin{proof} As in the proof of \thref{PROP:SAdistGENERAL}, we shall use the properties of $\omega$-capacity listed in \thref{LEM:CapacityProperties}. Take a decreasing sequence of compact sets $\{E_n\}_n$ which are finite unions of arcs, such that $E$ is contained in the interior of each $E_n$, and we have $\operatorname{Cap}_\omega(E_n) > 0$, $\bigcap_n E_n = E$ and
  \[
    \sum_{n \geq 0} \operatorname{Cap}_\omega(E_n)^{1/2} < \infty.
    \]
Let $\nu_n$ be an equilibrium measure for $E_n$ and $U_n = U_{\nu_n}$ be an equilibrium potential of $E_n$. Define $h_n := \operatorname{Cap}_\omega(E_n) U_n$. As in the proof of \thref{PROP:SAdistGENERAL}, on $\T$ we know that
    \begin{enumerate}
        \item $h_n \geq 1$ on $E_n$, possibly with the exception of a set of zero $\omega$-capacity,
        \item $0 \leq h_n$ on $\T$
        \item $\| h_n \|_{\ell^2(\Omega)} \leq \sqrt{\operatorname{Cap}_\omega(E_n)}$. 
    \end{enumerate}

    Then we can represent each $h_n$ as the Poisson integral of its boundary values in order to obtain harmonic extensions to $\D$. By property $(1)$ above combined with the fact that each $E_n$ is the union of finitely many arcs, we find $r_n \in (0,1)$ sufficiently close to $1$ so that
    \begin{equation}
        \label{E:hnLowerEstimateOnE}
        h_n(r_ne^{iv}) \geq 1/2 \quad \text{for all } e^{iv} \in E.
    \end{equation}
     Then $z \mapsto h_n(r_nz)$ is harmonic in a disk larger than $\overline{\D}$. Let $\widetilde{h_n}$ be the harmonic conjugate of $h_n$ on $\D$ that vanishes at the origin, and set
    \[
    \varphi_n(z) := h_n(r_nz) + i \widetilde{h_n}(r_nz), \quad z \in \overline{\D}.
    \]
    Since $|\widehat{h_n}(k)| = |\widehat{\widetilde{h_n}}(k)|$ for all integers $k \neq 0$, and the radial dilation only decreases the Fourier coefficients, the power series norm of $\ell^2_A(\Omega)$ shows that
    \[
    \| \varphi_n\|_{\ell^2(\Omega)} \leq 2 \| h_n\|_{\ell^2(\Omega)}  \leq 2\sqrt{\operatorname{Cap}_\omega(E_n)}.
    \]
    Thus, the function $\Phi := \sum_{n \geq 0} \varphi_n$ belongs to $\ell_A^2(\Omega)$ in view of the estimate
    \[
    \| \sum_{n \geq 0} \varphi_n \|_{\ell^2(\Omega)} \leq 2\sum_{n \geq 0}\sqrt{\operatorname{Cap}_\omega(E_n)} < \infty.
    \]    
    By the non-negativity in $(2)$ above, and by the bound from below in \eqref{E:hnLowerEstimateOnE}, we get for any integer $N>0$, and any point $\zeta \in E$ that
    \[ 
    \liminf_{z \to \zeta} \Re \Phi(z) \geq \liminf_{z \to \zeta} \sum_{n = 0}^N h_n(r_n z) \geq N/2.
    \]
    
    What remains to be proved is the analyticity of $\Phi$ near each point of $\T \setminus E$. Explicitly, we have
      \[
    \varphi_n(z) = \operatorname{Cap}_\omega(E_n) \int_\T K_W^A(r_n z,\zeta) d \nu_n(\zeta)
    \] where $K_W^A$ denotes  the analytic kernel defined in \eqref{E:KWADEF}. Since $z \mapsto K_W^A(z,\zeta)$ is analytic in the slit plane $\mathbb{C} \setminus \{ \zeta/r : r \in (0,1]\}$ and since $\supp{\nu_n} \subseteq E_n$, each $\phi_n$ is analytic $\mathbb{C} \setminus \{\zeta/rr_n : r \in (0,1], \zeta \in E_n\}$. 
    
    Fix $z_0 \in \T \setminus E$. For $N$ sufficiently large, there exists a small disk $D(z_0)$ centered at $z_0$ and $\delta > 0$ such that
    $|\zeta - zrr_n| > \delta$ for all $z \in D(z_0)$, $r \in (0,1]$ and all $\zeta \in E_n$, $n \geq N$. By \thref{Lemma: K_W^A bound}
    \[
    |\varphi_n(z)|  = \operatorname{Cap}_\omega(E_n) \int_{E_n} |K_W^A(r_n z,\zeta)| d \nu_n(\zeta) \leq C_\delta \operatorname{Cap}_\omega(E_n)
    \]
    for all $z \in D(z_0)$. Since $\sum_{n=0}^\infty \operatorname{Cap}_\omega(E_n) < \infty$, Weierstrass M-test implies that $\sum_{n > N}^\infty \varphi_n(z)$ converges uniformly to an analytic function on $D(z_0)$. Thus $\Phi(z) := \sum_{n\geq 0} \phi_n(z)$ is analytic in a neighborhood of $z_0$, and since $z_0 \in \T \setminus E$ was arbitrary, it follows that $\Phi$ extends analytically across every arc in $\T \setminus E$.
\end{proof}

In our second lemma, we shall utilize \thref{L:BLOWUPLEMMA} in order to construct a cyclic function in $\ell^2_A(\Omega)$ that vanishes precisely on $E$.

\begin{lemma}
    \thlabel{L:CyclicFunctionLemma}
    Let $E \subset \T$ be a compact set of zero $\omega$-capacity. There exists an outer function $f$ belonging to $C_A(\T) \cap \ell_A^2(\Omega)$ with $f=0$ on $E$, such that
    \[
    \{f(z)z^n: n=0,1,2,\dots \}
    \]
    has a dense linear span in $\ell^2_A(\Omega)$. Moreover, $f$ can be chosen to have an analytic continuation to a neighborhood of every point of $\T \setminus E$.
 \end{lemma}

 \begin{proof}
     Let $\Phi$ be as in \thref{L:BLOWUPLEMMA} and set
     \[ 
     f(z) = \exp(-\Phi(z)), \qquad z\in \cD.
     \] 
     Then $f: \D \to \D$ is analytic, and it has an analytic continuation to a neighborhood of each point in $\T \setminus E$, since $\Phi$ has this property. On the other hand, we also have that $\lim_{z \to \zeta} |f(z)| = 0$ whenever $\zeta \in E$. Hence we see that $f$ is continuous in $\cD$, and it vanishes precisely on $E$. Furthermore, it follows from \thref{LEM:ExpLemmaDouglas} in conjunction with \thref{L:BLOWUPLEMMA} that $f \in \ell^2_A(\Omega)$. 
     
     It now only remains to show that the linear span of 
     \[
     \{f(z)z^n: n=0,1,2,\dots \} \qquad \text{is dense in} \quad \ell^2_A(\Omega).
     \]
    According to \thref{L:BLOWUPLEMMA} we can ensure the representation $\Phi(z) = \sum_{n \geq 0} \phi_n(z)$ with each $\phi_n$ analytic in a disk larger than $\cD$ and with non-negative real part. Consider the functions
    \[
    F_N(z) := \exp \left( \sum_{n=0}^N \varphi_n(z) \right),  
    \]
    which are analytic in a disk larger than $\cD$. Let $p_{N,k}$ be the $k$-th order Taylor polynomial of $F_N$. Set $g_k := F_N - p_{N,k}$. Then $g_k(\zeta)$ and $g'_k(\zeta)$ tend uniformly to $0$ for $\zeta \in \T$. Since
    \[ |f(e^{it})g_k(e^{it}) - f(e^{is})g_k(e^{is})| \leq |f(e^{it})||g_k(e^{it}) - g_k(e^{is})| + |g_k(e^{is})||f(e^{it}) - f(e^{is})|\] and $|f| \leq 1$ on $\T$,
    we obtain from Douglas formula that
    \[ 
     \| fg_k\|_{\ell^2(\Omega)} = \|f(F_N - p_{N,k})\|_{\ell^2(\Omega)}  \to 0.
    \]
    As a consequence, we conclude that for any $N\geq 1$: 
    \begin{equation}
        \label{E:fFNstruct}
        f(z)F_N(z) = \exp \left(- \sum_{n > N} \varphi_n(z) \right)
    \end{equation} 
 belongs to the closed linear span of $\{f(z)z^n: n=0,1,2,\dots \}$ in $\ell^2_A(\Omega)$. Now applying  \thref{LEM:ExpLemmaDouglas} to the functions \eqref{E:fFNstruct}, we also get that
    \[
    \limsup_{N \to \infty} \|fF_N - f(0)F_N(0)\|_{\ell^2(\Omega)} \leq \limsup_{N \to \infty} C \|\sum_{n > N} \phi_n\|_{\ell^2(\Omega)} \leq \limsup_{N \to \infty} C\sum_{n>N} \norm{\phi_n}_{\ell^2(\Omega)} = 0.
    \]
    Since $f(0)F_N(0) \to 1$, we actually conclude that $fF_N \to 1$ in $\ell^2(\Omega)$. This completes the proof.
 \end{proof}

Most of the preparatory work towards proving \thref{Thm: four equivalent properties} has already been carried out in the previous lemmas, hence it only remains to assemble the final pieces.

\begin{proof}[Proof of \thref{Thm: four equivalent properties}]

  \proofpart{1}{$(i) \Rightarrow (ii)$.} This has already been established as a part of \thref{L:CyclicFunctionLemma}.

    \proofpart{2}{$(ii) \Rightarrow (iii)$.} Fix an arbitrary $h \in \ell^2_A(\Omega)$, and let $f$ be the cyclic function from part $(ii)$ which is continuous on $\cD$ and vanishes precisely on $E$. Then we can find analytic polynomials $(p_n)_n$ such that $p_n f \to h$ in $\ell^2(\Omega)$. Note that the $p_n f$ also vanishes on $E$. For each $n$, we let $F_n$ denote a Fej\'er polynomial of $p_nf$ of sufficiently high degree, so that $F_n \to 0$ uniformly on $E$ and $F_n \to h$ in $\ell^2(\Omega)$. We conclude that part $(iii)$ of the theorem holds in the special case when $g \in C(E)$ is identically zero.    

    Next, assume that $g \in C(E)$ is non-trivial. Let $(q_n)_n$ be a sequence of analytic polynomials converging uniformly to $g$ on $E$. This is possible, since $E$ has Lebesgue measure zero. By what we just proved for $h$, we may for each $n$ find a sequence $(F_{n,k})_k$ of analytic polynomials such that $F_{n,k} \to q_n$ in $\ell^2(\Omega)$ and $F_{n,k} \to 0$ uniformly on $E$. Then taking
    $Q_n := q_n - F_{n,k(n)}$ for sufficiently large $k(n)$, we obtain a sequence $(Q_n)_n$ of polynomials converging to $0$ in $\ell^2(\Omega)$ and uniformly to $g$ on $E$. It follows that $(iii)$ holds in the special case that $h \in \ell^2(\Omega)$ is identically zero. 
    
    Combining the results of the two paragraphs shows that the statement $(iii)$ holds in general.

    \proofpart{3}{$(iii) \Rightarrow (i)$.}
        Let $h=1 \in \ell^2(\Omega)$ and $g=0 \in C(E)$. According to the assumption in $(iii)$, there exist analytic polynomials $Q_n$ such that
    \[
    \|Q_n - 1 \|_{\ell^2(\Omega)} + \sup_{\zeta \in E} |Q_n(\zeta)| \rightarrow 0.
    \]
    Set $\varphi_n := 1-Q_n$, and note that $\|\varphi_n \|_{\ell^2(\Omega)}\rightarrow 0$ as $n \rightarrow \infty$.

    For the sake of obtaining a contradiction, assume that $\operatorname{Cap}_\omega(E) > 0$, and let $\mu$ be a probability measure supported on $E$ with $\|\mu\|_{\ell^2(\omega)} < \infty$. Then, by regarding $\mu$ as a bounded linear functional on $\ell^2(\Omega)$ through the Cauchy pairing (see Section~\ref{SEC:CauchyDualitySec}) and using the Cauchy-Schwarz inequality, we see that
    \[
    \left| \int_E \varphi_n d \mu \right| \leq  \|\mu\|_{\ell^2(\omega)} \|\varphi_n\|_{\ell^2(\Omega)} \rightarrow 0, \qquad \text{as} \quad n\to \infty.
    \]
    However, the uniform convergence of $Q_n$ to zero on $E$ implies
    \[
    \int_E \varphi_n d \mu = 1-\int_E Q_n d\mu \rightarrow 1, \qquad \text{as} \quad n\to \infty.
    \]
    From this contradiction, we conclude that $\operatorname{Cap}_\omega(E) = 0$, which completes the proof of the final implication.
\end{proof}

\begin{remark}
    Consider again the analytic function $\Phi$ appearing in \thref{L:BLOWUPLEMMA} which satisfies $\Re \Phi(z) \to +\infty$ as $z$ tends to any point of $E$ and is analytic in a neighborhood of any point of $\T \setminus E$. It follows that the function 
    \[ 
    P(z) := \frac{\Phi(z)}{\Phi(z) + A} \quad A > 0,
    \] extends continuously to $\cD$ and satisfies $|P(z)| \leq 1$ on $\T$, with equality if and only if $z \in E$. In fact we have $P(z) = 1$ on $E$. Moreover, a short argument involving the Douglas formula for $\ell^2(\Omega)$ in \thref{THM:DOUGLASFORMULA} shows that $P \in \ell^2(\Omega)$. Hence $P \in \ell^2(\Omega)$ is a so-called \textit{peak function} for $E$. 
    
    Applying the Douglas formula, a truncation argument for the series defining $\Phi$, and by means of further increasing the constant $A>0$ in the definition of $P$, one can actually produce such a peak function $P$ with arbitrarily small $\ell^2(\Omega)$-norm. We wish to mention that these functions can be used to constructively prove a peak interpolation result for $\ell^2(\Omega)$: for any compact set $E \subset \T$ of zero $\omega$-capacity and a function $g$ continuous on $E$, there exists a function $f \in \ell^2_A(\Omega)$ which extends continuously to $\cD$ and satisfies $f(z) = g(z)$ for $z \in E$, $\sup_{z \in \T} |f(z)| \leq \sup_{z \in E} |g(z)|$. See a similar argument in \cite[Theorem 6.1.3]{rudin1969function} which can be adapted to prove the mentioned result.
\end{remark}

In passing, we also mention that recent result on peak interpolation with Banach algebra functions in Drury-Arveson spaces was also considered in \cite{chalmoukis2022totally}.

Next, we prove \thref{THM:UNILAT-BILAT}, which states that if a compact subset $E \subset \T$ supports a measure that satisfies a one-sided estimate, then it must also support a measure that satisfies a two-sided estimate.

\begin{proof}[Proof of \thref{THM:UNILAT-BILAT}]

We will prove the contrapositive. Suppose that $\operatorname{Cap}_\omega(E)=0$
and let $\mu$ be a finite measure supported on $E$ satisfying
\eqref{EQ:OneSided}. We will prove that $\mu=0$.

By \eqref{EQ:OneSided}, the sequence
$\{\widehat{\mu}(n)\}_{n\geq0}$ defines a bounded linear functional on
$\ell_A^2(\Omega)$ by
\[
\ell_\mu(g)
:=
\sum_{n\geq0}
\widehat g(n)\,\overline{\widehat\mu(n)}.
\]
Moreover, whenever $g\in C_A(\mathbb T)\cap\ell_A^2(\Omega)$,
we have
\[
\ell_\mu(g)
=
\int_{\mathbb T}g(\zeta)\,\overline{d\mu(\zeta)}.
\]

Let $F\in C_A(\mathbb T)\cap\ell_A^2(\Omega)$ be the cyclic function
constructed in \thref{L:CyclicFunctionLemma} satisfying $F|_E=0$.
For every $k\geq0$, the boundedness of the shift gives
\[
z^kF\in C_A(\mathbb T)\cap\ell_A^2(\Omega).
\]
Since $\mu$ is supported on $E$, it follows that
\[
\ell_\mu(z^kF)
=
\int_E \zeta^k F(\zeta)\,\overline{d\mu(\zeta)}
=0.
\]

Thus $\ell_\mu$ vanishes on $\operatorname{span}\{z^kF:k\geq0\}$.
Since $F$ is cyclic, this span is dense in
$\ell_A^2(\Omega)$ and thus $\ell_\mu\equiv0$ on $\ell_A^2(\Omega)$.
In particular, applying $\ell_\mu$ to the monomials gives
\[
\widehat\mu(n)
=
\conj{\ell_\mu(z^n)}
=0,
\qquad n\geq0.
\]

The F. and M. Riesz theorem now implies that $\mu$ is absolutely continuous with respect to Lebesgue measure.
On the other hand $\operatorname{Cap}_\omega(E)=0 \implies |E|=0$.
Since $\mu$ is supported on $E$, we conclude that $\mu=0$, which finishes the proof \qedhere
\end{proof}


  \subsection{Translates in $\ell^2(\Omega)$}

  Here we prove \thref{THM:BICYCLIC} on the role of $\omega$-capacity for problems of translates in $\ell^2(\Omega)$.

  \begin{proof}[Proof of \thref{THM:BICYCLIC}]
  We shall establish the equivalence between $(i)$ and $(ii)$, and between $(i)$ and $(iii)$.

\proofpart{1}{$(i)$ implies $(ii)$:} Suppose that $E$ has zero $\omega$-capacity, and fix
$f\in C(\T)\cap\ell^2(\Omega)$ with 
\[
\{\zeta \in \T: f(\zeta)=0\} \subseteq E.
\]
Our primary aim is to establish the containment
\begin{equation}\label{EQ:Cinftyin[f]}
C^\infty_0(\T\setminus E)\subseteq [f],
\end{equation}
where $[f]$ denotes the closed linear span of the translates of $f$
in $\ell^2(\Omega)$.

Fix $\phi \in C^\infty_0(\T \setminus E)$ and note that since $f$ is continuous and does not vanish on $\supp{\phi}$, we can find a number $\delta>0$ such that
$|f|\geq\delta$ on $\supp{\phi}$. Now choose a Lipschitz function $\Phi:\C \to \C$ which agrees with $\frac{1}{z}$ whenever $|z|\geq \delta$. For instance, one may explicitly take
\[
\Phi(z) := \frac{\conj{z}}{\max\left( \, |z|^2, \, \delta^2 \right)} \qquad z\in \C.
\]
Now consider the function
\[
h(\zeta) =\Phi \circ f(\zeta) \cdot \phi(\zeta), \qquad \zeta \in \T,
\]
and note that $h$ has the remarkable property that
\[
\phi = f h \quad \text{on} \quad \T.
\]
Note that the claim essentially follows once we show that $h \in C(\T) \cap \ell^2(\Omega)$, and pass to appropriate Fej\'er means of $h$. To clarify this point, let $\mathcal{E}_\Omega(g)$ denote the Douglas integral of $g$ representing the semi-norm of $\ell^2(\Omega)$ as stated in \thref{THM:DOUGLASFORMULA}, and observe that
\[
\mathcal{E}_\Omega(h) \leq 2\norm{\Phi \circ f}_{L^\infty(\T)}^2\, \mathcal{E}_\Omega(\phi) + 2\norm{\phi}_{L^\infty(\T)}^2\,\mathcal{E}_\Omega(\Phi \circ f) < \infty.
\]
Using the observation in \eqref{EQ:Lipl2Omega}, it follows that $\Phi \circ f \in C(\T) \cap \ell^2(\Omega)$, and hence we conclude that $h \in C(\T) \cap \ell^2(\Omega)$. This conclusion allows us to exhibit the Fej\'er means $T_j$ of $h$, which satisfy
\[
\sup_{\zeta \in \T} \abs{T_j(\zeta)-h(\zeta)} \to 0, \qquad \norm{T_j-h}_{\ell^2(\Omega)} \to 0.
\]
Using the above product estimate once more, we obtain
\[
\mathcal{E}_{\Omega}\left(f(T_j-h)\right) \leq 2\norm{f}_{L^\infty(\T)}^2\,\mathcal{E}_{\Omega} (T_j-h) +2\norm{T_j-h}_{L^\infty(\T)}^2 \,\mathcal{E}_{\Omega}(f) \to 0.
\]
Since the $L^2(\T)$-norm is also easily shown to tend to zero, we conclude that $f T_j \to fh =\phi$ in
$\ell^2(\Omega)$. It follows that $\phi \in [f]$, which establishes the claim in \eqref{EQ:Cinftyin[f]}.

We can now carry out the proof of the desired implication. To this end, assume that $s\in\ell^2(\omega)$ is an arbitrary element which annihilates all translates of $f$, that is
\[
\sum_{n\in \mathbb{Z}} \widehat{s}(n) \conj{\widehat{f}(n-k)} =0, \qquad k\in \mathbb{Z}.
\]
Then the inclusion in \eqref{EQ:Cinftyin[f]} that was just established, coupled with the fact that $\ell^2(\omega)$ only consists of distributions, implies that we may interpret the annihilation in the distributional pairing as
\[
s(\phi) =0,\qquad \phi\in C^\infty_0(\T\setminus E).
\]
It then follows that $\supp{s}\subseteq E$. But since $E$ has zero $\omega$-capacity, \thref{THM:Distrib} forces $s\equiv0$. We therefore conclude that $f$ is bicyclic in $\ell^2(\Omega)$.
  
\proofpart{2}{$(ii)$ implies $(i)$.}
  Suppose $E$ has positive $\omega$-capacity, and let $\mu\in M(E)$ be a probability measure with $\{\widehat{\mu}(n)\}_n \in \ell^2(\omega)$. Fix any $f \in C(\T)\cap \ell^2(\Omega)$ with zero set $\{\zeta \in \T: f(\zeta)=0\} =E$. Then it follows that
  \[
  \sum_{n\in \mathbb{Z}} \widehat{f}(n-k) \conj{\widehat{\mu}(n)} = \int_{\T} f(\zeta) \zeta^k d\conj{\mu(\zeta)} =0, \qquad k\in \mathbb{Z}.
  \]
  But then $\mu$ is a non-trivial measure belonging to $\ell^2(\omega)$ which annihilates all translates of $f$, and hence $f$ cannot be bicyclic in $\ell^2(\Omega)$.

\proofpart{3}{$(i)$ equivalent to $(iii)$.} 
Suppose that $C^\infty_0(\T \setminus E)$ is dense in $\ell^2(\Omega)$ and assume that $\mu \in M(E) \cap \ell^2(\omega)$. But then 
\[
\int_{\T} \phi d\mu = 0, \qquad \phi \in C^\infty_0(\T \setminus E),
\]
from which the density assumption implies that $\mu \equiv 0$. This forces $E$ to have zero $\omega$-capacity. Conversely, if $C^\infty_0(\T \setminus E)$ is not dense in $\ell^2(\Omega)$, then by the Hahn-Banach Theorem, there exists a non-trivial element $s \in \ell^2(\omega)$ such that 
\[
\sum_{n} \widehat{s}(n) \conj{\widehat{\phi}(n)} =0, \qquad \forall \phi \in C^\infty_0(\T \setminus E).
\]
Since elements in $\ell^2(\omega)$ are distributions, we can re-express the above relation as the distributional pairing
\[
s(\conj{\phi})= 0, \qquad \forall \phi \in C^\infty_0(\T \setminus E),
\]
which forces $\supp{s} \subseteq E$. According to \thref{THM:Distrib} $E$ also supports a probability measure in $\ell^2(\omega)$, hence we conclude that $E$ has positive $\omega$-capacity.
\end{proof}

  We mention that the theory of Richter, Ross and Sundberg in \cite{ross1994hyperinvariant} on a complete characterizations of all translation invariant subspaces likely carries over to our setting of $\ell^2(\Omega)$. However, this lies outside the scope of our current work.

\section{Uniqueness sets for Dirichlet-type functions}\label{SEC:UNIQDIRICHLET}

\subsection{A capacitary characterization}


Here we devote our attention to proving \thref{THM:SUPPDIRICHLET}.
First, we shall need the following lemma.

\begin{lemma} \thlabel{LEM:SUPPDense} A compact set $E\subseteq \T$ supports a non-trivial $f\in L^2(E)$ with $\{\widehat{f}(n)\}_{n\in \mathbb{Z}}\in \ell^2(\Omega)$ if and only if 
\[
C_0^\infty(\T \setminus E) \quad \text{is not dense in} \quad \ell^2(\omega).
\]
\end{lemma}
\begin{proof} Assume $C_0^\infty(\T\setminus E)$ is dense in $\ell^2(\omega)$ and let $f\in L^2(E)$ with $\{\widehat{f}(n)\}_{n\in \mathbb{Z}}$ be arbitrary. Then since
\[
\int_{\T} f \conj{\phi} dm = \langle f, \phi \rangle =0, \qquad \phi \in C^\infty_0(\T \setminus E),
\]
it follows that $f\equiv 0$. Conversely, if $C^\infty_0(\T \setminus E)$ is not dense in $\ell^2(\omega)$, then by Hahn-Banach there must exist a non-trivial $L^2$-function $f \in \ell^2(\Omega)$ such that
\[
\int_{\T} f \conj{\phi} dm = \langle f, \phi \rangle =0,  \qquad \phi \in C^\infty_0(\T \setminus E).
\]
Regarding $fdm$ as a measure on $\T$, it then follows that $\supp{fdm}\subseteq E$, and hence $f$ is supported in $E$ and $\{\widehat{f}(n)\}_{n\in \mathbb{Z}}\in \ell^2(\Omega)$.
\end{proof}

Our next lemma is interesting in its own right and asserts that the linear span of all equilibrium measures associated with the collection of arcs in $\T$ is dense in $\ell^2(\omega)$.

\begin{lemma}\thlabel{LEM:EQUILIBDENSE} For any closed arc $I\subseteq \T$, let $\mu_I \in \ell^2(\omega)$ denote the equilibrium measure associated to $I$. Then the linear span of the set 
\[
\mathcal{P}_{\omega}:=\left\{ \mu_I: I\subseteq \T \right\}
\]
forms a dense subset of $\ell^2(\omega)$.
\end{lemma}
\begin{proof}
Arguing by duality, assume there exists $f\in \ell^2(\Omega)$ which annihilates the linear span of $\Po_\omega$:
\[
\sum_n \widehat{f}(n) \widehat{\mu_I}(-n)=0, \qquad I \subseteq \T.
\]
We must prove that $f\equiv 0$. The trick is that the notion of $\omega$-energy is translation-invariant, hence by uniqueness of equilibrium measures, we have that the translated measure $\tau_\theta \mu_I$ by $\theta$ is the equilibrium measure of the translated set $e^{i\theta}I$. For $0<r\leq 1$, let $I_r\subseteq \T$ denote the arc centered at $1$ of length $r$, and note that 
\[
\widehat{\mu_{e^{i\theta}I_r}}= e^{-in\theta} \widehat{\mu_{I_r}}(n),
\]
Then the annihilation condition of $f$ turns into 
\[
\sum_n \widehat{f}(n) \widehat{\mu_{I_r}}(-n) e^{-in\theta}=0, \qquad \theta \in [0,2\pi], \qquad 0<r<1.
\]
This means that the functions
\[
F_r(e^{i\theta}) := \sum_n \widehat{f}(n) \widehat{\mu_{I_r}}(-n)e^{-in\theta}, \qquad e^{i\theta} \in \T, \qquad 0<r<1
\]
belong to the Wiener algebra and vanish identically on $\T$. Indeed, we have 
\[
\sum_n \abs{\widehat{F_r}(n)} = \sum_n \abs{\widehat{f}(n) \widehat{\mu_{I_r}}(-n)} \leq \norm{f}_{\ell^2(\Omega)} \norm{\mu_{I_r}}_{\ell^2(\omega)}.
\]
We therefore conclude that $\widehat{f}(n) \widehat{\mu_{I_r}}(-n)=0$ for all $n\in \mathbb{Z}$ and $0<r\leq 1$. Since all the $\mu_{I_r}$ are probability measures, we must have $\widehat{f}(0)=0$. Next, fix an arbitrary integer $n\neq 0$, and pick $r>0$ with $r|n| < \pi $ and note that since
\[
\Re \widehat{\mu_{I_r}}(\pm n) = \int_{-r^/2}^{r/2} \cos(nt) d\mu_{I_r}(t) \geq \cos(rn/2) \mu_{I_r}(\T) >0,
\]
we have that $\widehat{\mu_{I_r}}(\pm n)\neq 0$, which at its turn forces $\widehat{f}(n)=0$. We therefore conclude that $\widehat{f}(n)=0$ for all $n$, hence $f=0$, and it follows that $\Po_\omega$ is dense in $\ell^2(\omega)$.
\end{proof}

We shall need one final lemma on smooth approximation of equilibrium measures from "within". 

\begin{lemma}\thlabel{LEM:APPROXBELOW} Assume that $C^\infty_0(\T \setminus E)$ is dense in $\ell^2(\omega)$. Then for any arc $I \subseteq \T$, there exist $\{\phi_j\}_j$ in $C^\infty_0(I \setminus E)$ such that $\phi_j dm$ are probability measures, and such that 
\[
\norm{\phi_j - \mu_I}_{\ell^2(\omega)} \to 0,
\]
where $\mu_I$ denotes the $\omega$-equilibrium measure of $I$.
\end{lemma}
\begin{proof}
\proofpart{1}{Local approximation:}
We first claim that for any closed arc $I \subseteq \T$, the closure of $C^\infty_0(I\setminus E)$ in the norm of $\ell^2(\omega)$ equals 
\[
S_I:=\{S\in \ell^2(\omega): \, \supp{S} \subseteq I \}.
\]
Since $C^\infty_0(I)$ is easily seen to be dense in $S_I$, it suffices to show that any $\phi \in C^\infty_0(I)$ can be approximated by functions in $C^\infty_0(I\setminus E)$. However, this essentially follows from the assumption that $C^\infty_0(\T\setminus E)$ is dense in $\ell^2(\omega)$. Indeed, approximate the constant $1$ by $(f_j)_j$ belonging to $C^\infty_0(\T \setminus E)$ in the $\ell^2(\omega)$-norm and multiply with any desirable function $\phi \in C^\infty_0(I)$. Then $\phi_j := \phi f_j \to \phi$ in $\ell^2(\omega)$, which settles the claim.
\proofpart{2}{A Hahn-Banach separation argument:}
Fix a closed arc $I\subseteq \T$ and consider the convex set
\[
\mathcal{C}_I := \{ f \in C_0(I \setminus E): fdm \quad \text{probability measure on} \, \, \T \}.
\]
Note if we can show that $\mu_I$ belongs to closure of $\mathcal{C}_I$ in the norm of $\ell^2(\omega)$, then a convolution with a standard approximation of identity also gives that we can pass $\phi_j \in C^\infty_0(I \setminus E)$ with $\phi_j dm$ probability measures such that
\[
\phi_j \to \mu_I \quad \text{in} \quad \ell^2(\omega).
\]

Arguing by contradiction, suppose this is not the case. Then by the Hahn-Banach separation Theorem there exists a real-valued linear functional $g\in \ell^2(\Omega)$ such that 
\[
\int_{I} g d\mu_I > 0, \qquad \int_{I} gf dm \leq 0, \qquad f\in \mathcal{C}_I.
\]
Plugging in approximates of the identities as $f$'s supported in $I\setminus E$, the second condition implies that, we must have $g\leq 0$ $dm$-a.e on $I \setminus E$. However, the condition $\int_{\T} gd\mu_I>0$  implies that the positive part $g_+$ of $g$ is non-trivial on $I$, and consequently we must have $g_+=0$ $dm$-a.e on $I \setminus E$, which readily implies that 
\[
\int_{I}g_+ f dm =0, \qquad f\in \mathcal{C}_I.
\]

Note that this forces $I \cap \supp{g_+ dm}\subseteq I \cap E$ and $g_+ \in \ell^2(\Omega)$ (the Douglas formula and that $\max{(f,g)} \in \ell^2(\Omega)$ whenever $f,g \in \ell^2(\Omega)$. Decomposing functions in $C_0(I \setminus E)$ into positive and negative parts (this is why we introduced $C_0(I\setminus E)$ in defining the set $\mathcal{C}_I$), we get
\[
\int_{\T} g_+ \phi dm =0, \qquad \phi \in C_0(I \setminus E).
\]
Since $C_0(I \setminus E)$ is dense in $S_I$ by step $1$, we conclude that $g_+ \in \ell^2(\Omega)$ which vanishes identically on $I$. This obviously contradicts that
\[
\int_{\T} g d\mu_I >0,
\] 
and thus establishes the claim.
\end{proof}
We are now ready to prove our main result.
\begin{proof}[Proof of \thref{THM:SUPPDIRICHLET}]
Note that according to \thref{LEM:SUPPDense}, it suffices to show the equivalence between $(ii)$, $(iii)$ and the statement: 
\begin{equation*}\tag{$i'$}
C_0^\infty(\T \setminus E) \quad \text{is dense in} \quad \ell^2(\omega).
\end{equation*}
\proofpart{1}{$(ii) \implies (i')$:}
We shall need the following useful observation. If $\mu_I$ is the $\omega$-equilibrium measure for $I$, then for any probability measure $\nu$ supported in $I$ that belongs to $\ell^2(\omega)$, we have
\[
\norm{(1-t)\mu_I+ t\nu}^2_{\ell^2(\omega)} \geq \norm{\mu_I}^2_{\ell^2(\omega)}, \qquad 0\leq t \leq 1.
\]
Expanding the product and minimizing the quadratic in $t$, we conclude that
\[
\Re \, \langle \mu_I, \nu \rangle_{\ell^2(\omega)}\geq \norm{\mu_I}^2_{\ell^2(\omega)},
\]
for any probability measure $\nu$ in $\ell^2(\omega)$, which in addition is supported in $I$. From this, we get that
\[
\norm{\mu_I- \nu}^2_{\ell^2(\omega)} \leq \norm{\nu}^2_{\ell^2(\omega)}-\norm{\mu_I}^2_{\ell^2(\omega)}.
\]
The above estimate implies that if condition $(ii)$ of the theorem holds, then we can find probability measures $\{\nu_j\}_j$ belonging to $\ell^2(\omega)$ and compactly supported in $I\setminus E$ such that
\[
\norm{\nu_j-\mu_I}_{\ell^2(\omega)} \to 0.
\]
Now by means of taking convolution of the $\nu_j$ by smooth approximates of the identity, and re-normalizing, we may assume that $\nu_j$ are also smooth on $\T$. As a consequence, it follows that the $\omega$-equilibrium measure of any arc $I\subseteq \T$ belongs to the closure of $C^\infty_0(\T\setminus E)$ in $\ell^2(\omega)$. The conclusion now follows from \thref{LEM:EQUILIBDENSE}.

\proofpart{2}{$(i') \implies (iii):$}
If $C^\infty_0(\T \setminus E)$ is dense in $\ell^2(\omega)$, then it follows from \thref{LEM:APPROXBELOW} that there exists $\{\phi_j\}_j$ in $C^\infty_0(\T \setminus E)$ such that $\phi_j dm$ are probability measures with 
\[
\norm{\phi_j -m}_{\ell^2(\omega)} \to 0
\]
where the unit-normalized Lebesgue measure $m$ is the equilibrium measure of $\T$.

\proofpart{3}{$(iii)\implies (ii)$:}
 Suppose that $\operatorname{Cap}_\omega(\T\setminus E)= \operatorname{Cap}_\omega(\T)$. Since the $\omega$-equilibrium measure of $\T$ is the Lebesgue measure $dm$, we have by assumption that there exist probability measures $\{\nu_j\}_j$ compactly supported in $\T\setminus E$, such that 
\[
\nu_j \to m \qquad{in} \quad \ell^2(\omega).
\]
Now convolving with smooth approximates to the identity and re-normalizing, we may in addition assume that $\nu_j$ are smooth on $\T$. Since multiplication by $\phi \in C^\infty(\T)$ is continuous on $\ell^2(\omega)$, it follows that
\[
\phi \nu_j \to \phi \quad \text{in} \quad \ell^2(\omega),
\]
hence $C^\infty(\T)$ belongs to the closure of $C^\infty_0(\T \setminus E)$ in $\ell^2(\omega)$. But since smooth functions are dense in $\ell^2(\omega)$, it follows that $C^\infty_0(\T \setminus E)$ is dense in $\ell^2(\omega)$. But appealing to \thref{LEM:APPROXBELOW} again, we can for any arc $I\subset \T$, find $\{\phi_j\}_j$ in $C^\infty_0(I\setminus E)$ such that $\phi_j dm$ are probability measures with 
\[
\norm{\phi_j-\mu_I}_{\ell^2(\omega)} \to 0.
\]
But then it follows that the condition $(ii)$ holds.

This completes the proof.
\end{proof}

\subsection{Unilateral versus bilateral uniqueness problem}

Here, we shall prove a discrepancy between the unilateral and the bilateral problem in \thref{THM:IMPKHRU}, as well as argue for why it is sharp. Our approach is based on a simultaneous approximation argument and is roughly summarized in the following proposition.

\begin{prop}\thlabel{PROP:SAUNIBI} Let $\{\omega(n)\}_{n\geq 0}$ be a non-increasing sequence of real numbers tending to $0$, such that 
\begin{equation}\label{EQ:omegalog}
\inf_{n > 1} \omega(n)\log n = 0.
\end{equation}
Then there exists a compact set $E \subset \T$ of positive Lebesgue measure such that the following conditions hold:
\begin{enumerate}
    \item[(i)] $C^\infty_0(\T \setminus E)$ is dense in $\ell^2(\omega)$,
    \item[(ii)] $E$ has finite Beurling--Carleson entropy:
    \[
    \sum_k |I_k| \log \frac{1}{|I_k|} < +\infty,
    \]
    where $\{I_k\}_k$ are the connected components of $\T\setminus E$.
\end{enumerate}    
\end{prop}

Note that \eqref{EQ:omegalog} is just a reformulation of the principal condition \eqref{EQ:OMEGACOND} in \thref{THM:IMPKHRU}. It follows from \thref{LEM:SUPPDense} that the condition $(i)$ is equivalent to the fact that for any non-trivial $f\in L^2(E)$ we have
\[
\sum_{n\in \mathbb{Z}} \abs{\widehat{f}(n)}^2 \Omega(|n|) =+ \infty,
\]
Meanwhile, the condition $(ii)$ implies that there exists a non-trivial $f\in L^2(E)$ whose Cauchy integral belongs to $C^\infty(\T)$. See \cite{khrushchev1978problem}. These observations imply that \thref{THM:IMPKHRU} is an immediate corollary of \thref{PROP:SAUNIBI}.


\begin{proof}[Proof of \thref{PROP:SAUNIBI}]
\proofpart{1}{Building-blocks:}
For any $0<\delta<1$, let $I_\delta \subset \T$ denote the open arc of length $\delta$ centered at $\zeta=1$.

Fix a smooth non-negative $C^\infty(\T)$-function $\phi$ with
\[
\int_{\T} \phi dm=1, \qquad \supp{\phi} \subseteq I_{1/2}.
\]
For numbers $0<\delta<1/2$, and integers $N>0$, we consider the family of smooth functions on $\T$ defined by
\[
\phi_{\delta,N}(\zeta) := \frac{1_{I_\delta}(\zeta^N)}{\delta} \phi(\zeta^{N/\delta}), \qquad \zeta \in \T.
\]
In other words, we have first localized $\phi$ by $\delta$, and then compose it with $\zeta^N$. The following properties are straightforward to verify:
\begin{enumerate}
    \item[(i)] $\int_{\T} \phi_{\delta,N}dm=1$,
    \item[(ii)] $\supp{\phi_{\delta,N}}$ is contained in the open set $U_{\delta,N}$, which is a union of $N$ open arcs, each of length $\delta/N$,
    \item[(iii)] $\widehat{\phi_{\delta,N}}(n)=0$ if $N\nmid n$ and there exists an absolute constant $C>0$ such that:
    \[
    \abs{\widehat{\phi_{\delta,N}}(Nn)} \leq C\min  \left( 1, \frac{1}{\delta |n|} \right), \qquad n\neq 0.
    \]
\end{enumerate}
Now for a suitable choice of parameters $\{\delta_j\}_j$ and integers $\{N_j\}_j$, the idea is to construct the compact subset $E\subset \T$ of the form:
\begin{equation}\label{DEF:E}
E= \bigcap_{j= 1}^\infty \T \setminus U_{\delta_j,N_j}.
\end{equation}
To estimate the Beurling--Carleson entropy of $E$, note that if $\{J_k\}_k$ are the connected components of $\T \setminus E$, and $\{I_{j}\}_{j}$ is an enumeration of the arcs forming $\{ U_{\delta_j,N_j}\}_j$, we have
\[
\sum_k |J_k| \log \frac{1}{|J_k|} \leq \sum_{k} \sum_{j:I_j \subseteq J_k} |I_j| \log \frac{1}{|J_k|} \leq \sum_{j}  |I_j|  \log \frac{1}{|I_j|} = \sum_j \delta_j \log \frac{1}{\delta_j/N_j}.
\]
We proceed with showing that $C_0^\infty(\T \setminus E)$ is dense in $\ell^2(\omega)$. Note that since the trigonometric polynomials are dense in $\ell^2(\omega)$ and $C_0^\infty(\T \setminus E)$ is translation invariant, it suffices to show that $1$ belongs to the closure of $C_0^\infty(\T \setminus E)$ in $\ell^2(\omega)$. To this end, we shall for brevity, set $\phi_j := \phi_{\delta_j,N_j}$, and note that each $\phi_j \in C^\infty_0(\T \setminus E)$. Furthermore, we have 
\begin{multline}\label{EQ:normest}
\norm{\phi_j-1}^2_{\ell^2(\omega)} = 2\sum_{n\geq 1} \abs{\widehat{\phi_j}(N_jn)}^2 \omega(N_jn) \\
\leq 2C^2 \sum_{n\geq 1} \min \left(1, \frac{1}{\delta^2 n^2} \right) \omega(N_jn)  \\ \leq 2C^2 \sum_{1\leq n \leq 1/\delta_j} \omega(N_j n) + \frac{1}{\delta_j^2}\sum_{n \geq 1/\delta_j} \frac{\omega(N_jn)}{n^2} \leq 4C^2 \frac{\omega(N_j)}{\delta_j}.
\end{multline}

Now using the assumption \eqref{EQ:omegalog}, we can find an increasing sequence of integers $\{N_j\}_j$ tending to infinity, such that
\[
\omega(N_j) \log N_j \leq 4^{-j}, \qquad j=1,2, \dots.
\]
Now set
\[
\delta_j := \frac{1}{3^j \lceil \log N_j \rceil}, \qquad j=1,2, \dots
\]
which readily satisfies
\[
\sum_j \delta_j <1, \qquad \sum_j \delta_j \log \frac{1}{\delta_j} < \infty.
\]

Using these parameters $\{N_j\}_j, \{\delta_j\}_j$, we form the corresponding compact set $E\subset \T$ as in \eqref{DEF:E}. Now it follows from \eqref{EQ:normest} that
\[
\norm{\phi_j-1}^2_{\ell^2(\omega)} \lesssim \frac{\omega(N_j)}{\delta_j} \lesssim 4^{-j} 3^j \to 0, \qquad j\to \infty.
\]
This implies that $C^\infty_0(\T \setminus E)$ is dense in $\ell^2(\omega)$, hence establishes $(i)$. Furthermore, since
\[
\sum_j \delta_j \log \frac{N_j}{\delta_j} \lesssim \sum_j \delta_j \log \frac{1}{\delta_j} + \sum_j 3^{-j} <\infty,
\]
we conclude that the corresponding set $E$ in \eqref{DEF:E} has positive Lebesgue measure, and finite Beurling--Carleson entropy. This proves $(ii)$, hence completes the proof. \qedhere

\end{proof}

In order to show that the condition \eqref{EQ:OMEGACOND} is sharp in \thref{THM:IMPKHRU}, it suffices to note that by Khrushchev's Theorem, if $E$ satisfies the property in $(ii)$ of \thref{THM:IMPKHRU}, then $E$ contains a compact subset $K$ of finite Beurling--Carleson entropy. Sharpness then follows from the following proposition.

\begin{prop}\thlabel{PROP:ENTROPYDECAY} Let $E\subset \T$ be a compact set of positive Lebesgue measure, and with finite Beurling--Carleson entropy. Then the indicator function $1_E$ satisfies the Fourier decay
\[
\sum_{n\in \mathbb{Z}} \log(1+|n|) \abs{\widehat{1_E}(n)}^2 < \infty.
\]

\end{prop}
\begin{proof}
Note that a simple application of Parseval and Fubini--Tonelli gives:
\[
\int_0^{1} \frac{\abs{(E+t)\Delta E}}{t}dt = \int_0^1 \int_{\T} \abs{1_E(\zeta e^{it})-1_E(\zeta)}^2 dm(\zeta) \frac{dt}{t} \asymp \sum_{n \neq 0} \abs{\widehat{1_E}(n)}^2 \int_0^{1} \frac{\sin^2(n\pi t)}{t}dt,
\]
where $E \Delta F = (E\setminus F) \cup (F \setminus E)$ denotes the symmetric difference of subsets $E,F\subset \T$. Using a change of variable, trigonometric formulas and decomposing the integral on appropriate intervals, one can easily show that
\[
\int_0^1\frac{\sin^2(n\pi t)}{t}dt = \int_0^{n\pi} \frac{\sin^2(t)}{t}dt \gtrsim \log (1+n), \qquad n\geq 1.
\]
We actually have comparability, but this is not needed here. Now using the simple estimate 
\[
\abs{(E+t)\Delta E} =2 |(E+t)\cap \T \setminus E| \leq 2 \sum_k \min (t, \abs{I_k}), 
\]
where $\{I_k\}_k$ are the connected components of $\T \setminus E$, we get
\[
\int_0^{1} \frac{\abs{(E+t)\Delta E}}{t}dt \leq 2 \sum_k \int_0^1 \frac{\min(t,|I_k|)}{t} dt = 2 \sum_k \abs{I_k}\left(1+ \log \frac{1}{|I_k|}\right)< \infty.
\]
This establishes the desired claim.
\end{proof}

\subsection{A simple proof of the Ahlfors--Beurling Theorem} \label{SSEC:BATHM}

Here, we exhibit an application of our developments in order to give a simple proof of the Ahlfors--Beurling \thref{THM:A-BTHM}. 

\begin{proof}[New proof of \thref{THM:A-BTHM}]
Let $\Omega(n)= (1+n)$, and note that in view of \thref{THM:SUPPDIRICHLET} it suffices to show that $E$ supports a non-trivial element $f \in \ell^2(\Omega)$ if and only if there exists a non-constant holomorphic function $F$ in $\C \setminus E$ with
\[
\int_{\C \setminus E } \abs{F'(z)}^2 dA(z) < \infty.
\]
\proofpart{1}{Cauchy integral extension:}
Suppose $E$ supports a (real-valued) $f \in \ell^2(\Omega)$ and consider the Cauchy integral
\[
F(z) := \Ka(f)(z) = \int_E \frac{\zeta f(\zeta)}{\zeta -z} dm(\zeta), \qquad z\in \C \setminus E.
\]
Then $F$ is clearly analytic in $\C \setminus E$. Now since 
\[
F(z) = \sum_{n\geq 0} \widehat{f}(n) z^n, \qquad |z|<1,
\]
we immediately see that the assumption $f\in \ell^2(\Omega)$ implies
\[
\int_{|z|<1} \abs{F'(z)}^2 dA(z) = \sum_{n\geq 0} n\abs{\widehat{f}(n)}^2< \infty.
\]
It thus only remains to show that
\[
\int_{|z|>1} \abs{F'(z)}^2 dA(z) < \infty.
\]
To this end, note that we have the identity:
\[
\Ka(f)'(z)  = \frac{1}{z^2} \conj{\Ka(f)'(1/\conj{z})}, \qquad  |z|>1.
\]
Hence, applying the change of variable $z \mapsto 1/\conj{z}$, we get
\begin{multline*}
\int_{|z|>1} |F'(z)|^2 dA(z) = \int_{|z|>1} |\Ka(f)'(1/\conj{z})|^2 \frac{dA(z)}{|z|^4} \\
= \int_{|\lambda|<1} |\Ka(f)'(\lambda)|^2 dA(\lambda) = \sum_{n\geq 0} n|\widehat{f}(n)|^2 < \infty.
\end{multline*}
This proves that $F$ has the required properties, hence $E$ is not removable for $W^{1,2}_a$.

\proofpart{2}{Fatou Jump Theorem:}
Conversely, let $F$ be a non-trivial element in $W^{1,2}$ which is analytic in $\C \setminus E$, and note that analyticity allows us to write
\[
F(z) = \sum_{n\geq 0} a_n z^n, \qquad |z|<1, \qquad F(z) = \sum_{n\geq 0} \frac{b_n}{z^n}, \qquad |z|>1. 
\]
With these expansions at hand, it is easy to see that the assumption $F\in W^{1,2}$ implies 
\[
\sum_{n\geq 0}n\left( |a_n|^2 + |b_n|^2\right) < \infty.
\]
Now consider the formal limit
\[
f(\zeta) := \lim_{r\to 1-} F(r\zeta) - F(\zeta/r), \qquad \zeta \in \T,
\]
which by the dominated convergence theorem defines an element in $\ell^2(\Omega)$. Furthermore, since $F$ is analytic in $\C \setminus E$, we also have that $\supp{fdm}\subseteq E$. If $f$ were to be zero, then since 
\[
\widehat{f}(n) = \begin{cases} a_n, \qquad n>0 \\
- b_n \qquad n<0
    
\end{cases}
\]
we would consequently have that $F=a_0$ on $|z|<1$ and $F=b_0$ on $|z|>1$. Now since $F$ is analytic across $\T \setminus E$ we must have $a_0=b_0$, and hence $F$ is constant contradicting the assumption that $F$ was non-constant. We have thus produced a non-trivial $f \in \ell^2(\Omega)$ which in addition is supported in $E$. 

\end{proof}

\subsection*{Declaration of AI-assistance} The mathematical content of this paper has progressively and solely been developed by the authors. At the final stages, LLMs were implemented in order to locate small errors and other minor inconsistencies.

\bibliographystyle{siam}
\bibliography{mybib}

\begin{thebibliography}{10}

\bibitem{ahlfors1950conformal}
{\sc L.~Ahlfors and A.~Beurling}, {\em {Conformal invariants and function-theoretic null-sets}}, Acta Mathematica, 83 (1950), pp.~101--129.

\bibitem{beneteau2020simultaneous}
{\sc C.~B{\'e}n{\'e}teau, O.~Ivrii, M.~Manolaki, and D.~Seco}, {\em {Simultaneous zero-free approximation and universal optimal polynomial approximants}}, Journal of Approximation Theory, 256 (2020), p.~105389.

\bibitem{beurling1949spectres}
{\sc A.~BEURLING}, {\em {Sur les spectres des fonctions\^{} Colloques internationaux}}, Centre National de la Recherche Scientifique, Paris,  (1949).

\bibitem{brown1984cyclic}
{\sc L.~Brown and A.~L. Shields}, {\em Cyclic vectors in the dirichlet space}, Trans. Amer. Math. Soc, 285 (1984).

\bibitem{carlesonuniqueness}
{\sc L.~Carleson}, {\em Sets of uniqueness for functions regular in the unit circle}, Acta mathematica, 87 (1952), pp.~325--345.

\bibitem{chalmoukis2022totally}
{\sc N.~Chalmoukis and M.~Hartz}, {\em {Totally null sets and capacity in Dirichlet type spaces}}, Journal of the London Mathematical Society, 106 (2022), pp.~2030--2049.

\bibitem{chalmoukis2024potential}
\leavevmode\vrule height 2pt depth -1.6pt width 23pt, {\em {Potential theory and boundary behavior in the Drury-Arveson space}}, arXiv preprint arXiv:2410.07773,  (2024).

\bibitem{deleeuw1970two}
{\sc K.~de~Leeuw and Y.~Katznelson}, {\em The two sides of a {F}ourier-{S}tieltjes transform and almost idempotent measures}, Israel Journal of Mathematics, 8 (1970), pp.~213--229.

\bibitem{dirichletspaceprimer}
{\sc O.~El-Fallah, K.~Kellay, J.~Mashreghi, and T.~Ransford}, {\em A primer on the {D}irichlet space}, vol.~203 of Cambridge Tracts in Mathematics, Cambridge University Press, Cambridge, 2014.

\bibitem{hedberg1974removable}
{\sc L.~I. Hedberg}, {\em {Removable singularities and condenser capacities}}, Arkiv f{\"o}r Matematik, 12 (1974), pp.~181--201.

\bibitem{khrushchev1978problem}
{\sc S.~V. Khrushchev}, {\em The problem of simultaneous approximation and of removal of the singularities of {C}auchy type integrals}, Trudy Matematicheskogo Instituta imeni VA Steklova, 130 (1978), pp.~124--195.

\bibitem{kozma2013singular}
{\sc G.~Kozma and A.~Olevskii}, {\em {Singular distributions, dimension of support, and symmetry of Fourier transform}}, Annales de l'Institut Fourier, 63 (2013), pp.~1205--1226.

\bibitem{LevOlevskiiPS}
{\sc N.~Lev and A.~Olevskii}, {\em Piatetski-{S}hapiro phenomenon in the uniqueness problem}, C. R. Math. Acad. Sci. Paris, 340 (2005), pp.~793--798.

\bibitem{lev2011wiener}
\leavevmode\vrule height 2pt depth -1.6pt width 23pt, {\em {Wiener's' closure of translates' problem and Piatetski-Shapiro's uniqueness phenomenon}}, Annals of Mathematics,  (2011), pp.~519--541.

\bibitem{limani2026asymmetric}
{\sc A.~Limani and T.~Persson}, {\em {Asymmetric uniqueness sets in $\ell^{q}$}}, arXiv preprint arXiv:2603.08482,  (2026).

\bibitem{makarov1991class}
{\sc N.~G. Makarov}, {\em {On a class of exceptional sets in the theory of conformal mappings}}, Mathematics of the USSR-Sbornik, 68 (1991), pp.~19--30.

\bibitem{MatillaFourier}
{\sc P.~Mattila}, {\em Fourier analysis and Hausdorff dimension}, vol.~150 of Cambridge Studies in Advanced Mathematics, Cambridge University Press, Cambridge, 2015.

\bibitem{Piatetski-shapiro}
{\sc {Piatetski-Shapiro, Ilya}}, {\em Selected works of {I}lya {P}iatetski-{S}hapiro}, American Mathematical Society, Providence, RI, 2000.
\newblock Edited and with commentaries by James Cogdell, Simon Gindikin, Peter Sarnak, Pierre Deligne, Stephen Gelbart, Roger Howe and Stephen Rallis.

\bibitem{ross1994hyperinvariant}
{\sc W.~T. Ross, S.~Richter, and C.~Sundberg}, {\em {Hyperinvariant subspaces of the harmonic Dirichlet space}}, Journal f{\"u}r die reine und angewandte Mathematik, 448 (1994), p.~1.

\bibitem{rudin1969function}
{\sc W.~Rudin}, {\em Function Theory in Polydiscs}, vol.~41 of Mathematics Lecture Note Series, W. A. Benjamin, New York, 1969.

\end{thebibliography}

\Addresses

\end{document}